\documentclass[12pt]{amsart}
\usepackage{amsfonts, amssymb, latexsym, epsfig, amsthm}
\usepackage[usenames,dvipsnames]{xcolor}
\usepackage{graphicx,pgf,tikz}
\usetikzlibrary{decorations.markings}
\usepackage{ytableau}
\usepackage[all]{xy}
\usepackage{wrapfig}
\usepackage{enumitem}

\newtheorem{theorem}{Theorem}[section]
\newtheorem{proposition}[theorem]{Proposition}
\newtheorem{lemma}[theorem]{Lemma}

\newtheorem{claim}[theorem]{Claim}

\newtheorem*{claim*}{Claim}
\newtheorem{corollary}[theorem]{Corollary}
\newtheorem{Main Conjecture}[theorem]{Main Conjecture}

\theoremstyle{remark}
\newtheorem{definition}[theorem]{Definition}
\newtheorem{example}[theorem]{Example}

\theoremstyle{plain}

\begin{document}
\pagestyle{plain}

\mbox{}
\title{Prism Tableaux for Alternating Sign Matrix Varieties}
\author{Anna Weigandt}
\address{Dept.~of Mathematics, U.~Illinois at
Urbana-Champaign, Urbana, IL 61801, USA}
\email{weigndt2@uiuc.edu}

\date{\today}

\maketitle

\begin{abstract}
A prism tableau is a set of reverse semistandard tableaux,  each  positioned within an ambient grid.  Prism tableaux were introduced to provide a formula for the Schubert polynomials of A.~Lascoux and M.P.~Sch\"utzenberger.  This formula directly generalizes the well known expression for Schur polynomials as a sum over semistandard tableaux.    Alternating sign matrix varieties generalize the matrix Schubert varieties of  W.~Fulton.  We use prism tableaux to give a  formula for the multidegree of an alternating sign matrix variety.  
\end{abstract}

\tableofcontents

\section{Introduction}
 An {\bf alternating sign matrix} (ASM) is a square matrix with entries in $\{-1,0,1\}$ so that 
\begin{enumerate}[label=(A\arabic*),ref=(A\arabic*)]
\item \label{item:A1} the nonzero entries in each row and column alternate in sign and 
\item \label{item:A2} each row and column sums to 1.
\end{enumerate}
Let ${\sf ASM}(n)$ be the set of all $n\times n$ ASMs.  
The enumeration of ASMs has drawn much interest, the 
sequence for $n\geq 1$ being
\[ 1, 2, 7, 42, 429, 7436, 218348, 10850216, 911835460,\ldots.\]
 There is a closed form expression for this sequence; the celebrated \emph{alternating sign matrix conjecture} of 
W.~H.~Mills--D.~P.~Robbins--H.~Rumsey \cite{mills1983alternating} asserts that
\[|{\sf ASM}(n)|=\prod_{j=0}^{n-1}\frac{(3j+1)!}{(n+j)!}.\]
The original proof  was given by D.~Zeilberger \cite{zeilberger1996proof}. A second proof was given by G.~Kuperberg \cite{kuperberg1996another} using  the six-vertex model of statistical mechanics. See \emph{Proofs and Confirmations: The Story of the Alternating-Sign Matrix Conjecture}, by D.~Bressoud, for the link between ASMs and hypergeometric series, plane partitions, and lattice paths \cite{bressoud1999proofs}.

  Each $A=(a_{ij})_{i,j=1}^n\in{\sf ASM}(n)$ has an associated {\bf corner sum function}
\begin{equation}
\label{eqn:cornersum}
r_A(i,j)=\sum_{k=1}^i\sum_{\ell =1}^ja_{k\ell}.
\end{equation}  Corner sum functions define a lattice structure on ${\sf ASM}(n)$; say
\begin{equation}
\label{eqn:asmposetdef}
A\leq B \text{ if and only if } r_A(i,j)\geq r_B(i,j) \text{ for all } 1\leq i,j\leq n.
\end{equation}
Restricted to permutation matrices, (\ref{eqn:asmposetdef}) is the {\bf Bruhat order} on the symmetric group $\mathcal S_n$.  A.~Lascoux and M.P.~Sch\"utzenberger showed that ${\sf ASM}(n)$ is the smallest lattice which contains  $\mathcal S_n$ as an order embedding  \cite{lascoux1996treillis}.

 A {\bf partition} is a weakly decreasing sequence of nonnegative integers \[\lambda=(\lambda_1, \lambda_2, \lambda_3, \ldots , \lambda_h).\]  The {\bf length} of $\lambda$ is  $\ell(\lambda):=|\{i:\lambda_i\neq 0\}|$. 
 Fix tuples  of partitions and positive integers
\begin{equation}
\label{eqn:tupledef}
 \boldsymbol \lambda= (\lambda^{(1)},\ldots, \lambda^{(k)}) \text{ and } \mathbf d=(d_1,\ldots,d_k) \text{ so that } d_i\geq \ell(\lambda^{(i)}) \text{ for all } i.
\end{equation}  
 We associate to each $(\boldsymbol \lambda,\mathbf d)$ an ASM, denoted
 $A_{\boldsymbol \lambda,\mathbf d}$, which is the least upper bound of a list of \emph{Grassmannian} permutations.
Conversely,  for any ASM, there exists some $(\boldsymbol \lambda,\mathbf d)$ so that $A=A_{\boldsymbol \lambda,\mathbf d}$.

Prism tableaux were first defined in \cite{weigandt2015prism}.  We give a more general definition here.
A \emph{prism tableau} for  $(\boldsymbol \lambda,\mathbf d)$ is a $k$-tuple of \emph{reverse semistandard tableaux}, with shapes and labels determined by the pair $(\boldsymbol \lambda,\mathbf d)$.   We write ${\tt Prism}(\boldsymbol \lambda,\mathbf d)$ for the set of \emph{minimal prism tableaux} for $(\boldsymbol \lambda,\mathbf d)$ which have no \emph{unstable triples}.   These terms are defined in Section~\ref{subsect:prism}.  Each prism tableau has an associated weight monomial ${\tt wt}(\mathcal T)$. 
Let
\begin{equation}
\label{eqn:prismgenseries}
\mathfrak A_{\boldsymbol \lambda,\mathbf d}=\sum_{\mathcal T\in {\tt Prism}(\boldsymbol \lambda,\mathbf d)}{\tt wt}(\mathcal T).
\end{equation}
Call $\mathfrak A_{\boldsymbol \lambda,\mathbf d}$ an {\bf ASM polynomial}.

If $\boldsymbol \lambda=(\lambda)$ and $\mathbf d=(d)$, the polynomial  $\mathfrak A_{\boldsymbol \lambda,\mathbf d}$ is the {\bf Schur polynomial} $s_{\lambda}(x_1,\ldots,x_d)$.  This follows immediately from the usual definition of $s_{\lambda}$ as a weighed sum over \emph{semistandard tableaux}.  The \emph{Schubert polynomials} $\{\mathfrak S_w:w\in \mathcal S_\infty\}$ of  A.~Lascoux and M.P.~Sch\"utzenberger \cite{Lascoux.Schutzenberger} generalize Schur polynomials.  The purpose of  \cite{weigandt2015prism} was to provide a prism formula for Schubert polynomials.  We prove the following generalization.
\begin{theorem}
\label{theorem:schubertsum}
$\displaystyle \mathfrak A_{\boldsymbol \lambda,\mathbf d}=\sum_{w\in{\tt MinPerm}(A_{\boldsymbol \lambda,\mathbf d})}\mathfrak S_w$.
\end{theorem}
  Here, ${\tt MinPerm}(A)$ denotes the set permutations above $A$ in ${\sf ASM}(n)$ which have the minimum possible length.  
Our proof of Theorem~\ref{theorem:schubertsum} is purely combinatorial; we give a bijection between ${\tt Prism}(\boldsymbol \lambda,\mathbf d)$ and the set of facets of the \emph{subword complexes} (\cite{knutson2004subword}) for each $w\in{\tt MinPerm}(A)$.  The Schubert polynomial is a weighted sum over the facets of its corresponding subword complex \cite{Fomin.Kirillov,Bergeron.Billey,knutson2005grobner}.

In Section~\ref{subsection:biject}, we define a map from the set of all prism tableaux to a simplicial complex $\Delta(Q_{n\times n},A)$, which is itself a union subword complexes. Restricted to ${\tt Prism}(\boldsymbol \lambda,\mathbf d)$, this map is a bijection onto the set of maximal dimensional facets in $\Delta(Q_{n\times n},A)$ (see Theorem~\ref{thm:wtbijection}).

 $\mathfrak A_{\boldsymbol \lambda,\mathbf d}$ also has a geometric interpretation; it is the \emph{multidegree} of an \emph{alternating sign matrix variety}.   
Write ${\sf Mat}(n)$ for the space of $n\times n$ matrices over an algebraically closed field $\Bbbk$.  Given $M\in {\sf Mat}(n)$, let  $M_{[i],[j]}$ be the submatrix of $M$ which consists of the first $i$ rows and $j$ columns of $M$.  
We define the {\bf alternating sign matrix variety} 
\begin{equation}
X_A:=\{M\in{\sf Mat}(n): {\rm rank}(M_{[i],[j]})\leq r_A(i,j) \text{ for all } 1\leq i,j\leq n\}.
\end{equation}
If $w\in \mathcal S_n$, then $X_w$ is a {\bf matrix Schubert variety} as defined in \cite{fulton1992flags}.  

ASM varieties are stable under multiplication by the group of  invertible, diagonal matrices $\sf T\subset {\sf GL}(n)$.  There is a corresponding $\mathbb Z^{n}$ grading and multidegree
\[\mathcal C(X_A;\mathbf x)\in \mathbb Z[x_1,\ldots,x_n].\]  Whenever $w\in \mathcal S_n$,  we have $\mathfrak S_w=\mathcal C(X_w;\mathbf x)$.  This was shown in \cite{knutson2005grobner} and is equivalent to earlier statements in the language of equivariant cohomology  \cite{feher2003schur} and degeneracy loci \cite{fulton1992flags}.
We show $\mathfrak A_{\boldsymbol \lambda,\mathbf d}$ is the multidegree of the ASM variety $X_{A_{\boldsymbol \lambda,\mathbf d}}$.
\begin{theorem}
\label{theorem:main}
Fix $\boldsymbol \lambda$ and $\mathbf d$ as in (\ref{eqn:tupledef}).   
Then
\[\mathcal C(X_{A_{\boldsymbol \lambda,\mathbf d}};\mathbf x)=\mathfrak A_{\boldsymbol \lambda,\mathbf d}.\]
\end{theorem}
The irreducible components of $X_A$ are always matrix Schubert varieties.   Theorem~
\ref{theorem:main} follows from Theorem~\ref{theorem:schubertsum} and the additivity of multidegrees.  

 We also discuss the explicit connection of prism tableaux to Gr\"obner geometry of $X_A$.
Let $Z=(z_{ij})_{i,j=1}^n$ be the generic $n\times n$ matrix.  
Define the {\bf ASM ideal} by
\begin{equation}
\label{def:detIdeal}
I_A:=\langle  \text{ minors of size } r_A(i,j)+1 \text{ in } Z_{[i],[j]}\rangle.
\end{equation}  It is immediate  that $I_A$ provides set-theoretic equations for $X_A$.  For any $A\in {\sf ASM}(n)$, we $I_A$ is radical.   This follows from the Frobenius splitting argument given in \cite[Section~7.2]{knutson2009frobenius}.  We make the connection to ASM varieties explicit.
\begin{proposition}[{\cite{knutson2009frobenius}}]
\label{proposition:maindegeneration}
Fix any antidiagonal term order $\prec$ on $\Bbbk[Z]$.
\begin{enumerate}
\item The essential (and hence defining) generators of $I_A$ form a Gr\"obner basis under  $\prec$.
\item $I_A$ is radical and its initial ideal is a square-free monomial ideal.
\item The Stanley-Reisner complex of ${\tt init}(X_A)$ is $\Delta(Q_{n\times n}, A)$.  
\end{enumerate}
\end{proposition}
Since ${\tt Prism}(\boldsymbol \lambda,\mathbf d)$ is in weight preserving bijection with the facets of maximum dimension in $\Delta(Q_{n\times n}, A)$, this yields a second proof of Theorem~\ref{theorem:main}.

\section{Prism tableaux and ASMs}

\label{section:prismandasm}

\subsection{Rothe diagrams for ASMs}

\label{subsection:rothe}
We start by presenting a generalization of Rothe diagrams to ASMs.   Following \cite{mills1983alternating}, say $A=(a_{i j})_{i,j=1}^n\in{\sf ASM}(n)$ has an {\bf inversion} in position $(i,j)$ if 
\begin{equation}
\label{eqn:inversion1}
\sum_{(k, l):i<k\text{ and }j<l}a_{i l}a_{k j}=1.
\end{equation}
Write $[n]:=\{1,\ldots,n\}$.  Then 
\begin{equation}
\label{eqn:asmdiagram}
D(A):=\{(i,j):(i,j) \text{ is an inversion of } A\}\subset n\times n
\end{equation} is the {\bf Rothe diagram} of $A$. We represent $D(A)$ graphically.  Our convention is to visually indicate the ASM by placing a black dot for each 1 in $A$ and a white dot for each $-1$.     
The {\bf essential set} $\mathcal Ess(A)$ consists of the southeast most corners of each connected component of $D(A)$,
 \[\mathcal Ess(A):=\{(i,j)\in D(A):(i+1,j),(i,j+1)\not \in D(A)\}.\]

\begin{example}
\label{example:asmDiagram}

\[ A=\left(\begin{matrix} 0 & 0 & 0 & 1\\ 0&1&0&0\\ 1&-1&1&0\\ 0 &1 & 0 & 0\end{matrix}\right)
\hspace{3em} 
D(A)=\begin{tikzpicture}[x=1.3em,y=1.3em,baseline=3.6em]
	\draw[step=1,gray!30,very thin] (0,1) grid (4,5); 
      \draw[color=black, thick](0,1)rectangle(4,5);
     \filldraw[color=black, fill=gray!30, thick](0,4)rectangle(1,5);
     \filldraw[color=black, fill=gray!30, thick](1,4)rectangle(2,5);
     \filldraw[color=black, fill=gray, thick](2,4)rectangle(3,5);
     \filldraw[color=black, fill=gray, thick](0,3)rectangle(1,4);
     \filldraw[color=black, fill=gray, thick](1,2)rectangle(2,3);
     \filldraw [black](3.5,4.5)circle(.1);
     \filldraw [black](1.5,3.5)circle(.1);
     \filldraw [black](.5,2.5)circle(.1);
     \filldraw [color=black,fill=white,thick](1.5,2.5)circle(.1);
     \filldraw [black](2.5,2.5)circle(.1);
     \filldraw [black](1.5,1.5)circle(.1);
     \draw[thick] (3.5,4.5)--(3.5,4);
     \draw[thick] (4,4.5)--(3.5,4.5);
     \draw[thick] (1.5,3.5)--(1.5,3);
     \draw[thick] (2,3.5)--(1.5,3.5);
     \draw[thick] (.5,2.5)--(.5,2);
     \draw[thick] (1,2.5)--(.5,2.5);
     \draw[thick] (2.5,2.5)--(2.5,2);
     \draw[thick] (3,2.5)--(2.5,2.5);
     \draw[thick] (1.5,1.5)--(1.5,1);
     \draw[thick] (2,1.5)--(1.5,1.5);
     \draw[thick] (3,3.5)--(2,3.5);
     \draw[thick] (4,3.5)--(3,3.5);
     \draw[thick] (4,2.5)--(3,2.5);
     \draw[thick] (3,1.5)--(2,1.5);
     \draw[thick] (4,1.5)--(3,1.5);
     \draw[thick] (3.5,4)--(3.5,3);
     \draw[thick] (3.5,3)--(3.5,2);
     \draw[thick] (.5,2)--(.5,1);
     \draw[thick] (2.5,2)--(2.5,1);
     \draw[thick] (3.5,2)--(3.5,1);
     \end{tikzpicture}\]
The boxes of the diagram of $A$ are shaded gray.  The essential boxes are dark gray. \qed
\end{example}

The sum in (\ref{eqn:inversion1}) factorizes 
\begin{equation}
\label{eqn:inversion2}
\sum_{(k,l):i<k\text{ and }j<l}a_{i l}a_{k j}=\left(\sum_{k=i+1}^na_{kj}\right)\left(\sum_{l=j+1}^na_{il}\right)=\left(1-\sum_{k=1}^ia_{kj}\right)\left(1-\sum_{l=1}^ja_{il}\right).
\end{equation}
See \cite{Bousquet.Habsieger}.  By conditions \ref{item:A1} and  \ref{item:A2} the factors in RHS of (\ref{eqn:inversion2}) product are always 0 or 1.  In order for $(i,j)$ to be an inversion, both must be 1. Visually, this amounts to striking out hooks to the right and below each black dot which stop just before they encounter a box which contains a white dot.  The boxes which remain are the elements of $D(A)$.

Notice that  $D(A)$ is similar to the ASM diagram defined by A.~Lascoux \cite{lascoux2008chern}.  However, our conventions on inversions differ; we include the set of \emph{negative inversions} in our diagram.
If $w$ is a permutation matrix, $D(w)$ and $\mathcal Ess(w)$ coincide with the usual Rothe diagram and essential set, as defined in \cite{fulton1992flags}.
Any permutation is uniquely determined by the restriction of the corner sum function to its essential set \cite[Lemma 3.10]{fulton1992flags}.  The same statement holds more generally for ASMs, see Proposition~\ref{prop:bigrassList}.

 Given $w\in \mathcal S_n$, the permutation matrix of $w$ is an $n\times n$ matrix with a one in each of the $(i,w(i))$ positions and zeros elsewhere.  This defines an embedding of $\mathcal S_n\hookrightarrow {\sf ASM}(n)$.  We freely identify each permutation with its permutation matrix. 
  The length $\ell(w)$ of $w\in \mathcal S_n$ is the number of inversions, or equivalently $\ell(w)=|D(w)|$. 
Say 
\begin{equation}
\label{eqn:degA}
{\tt deg}(A)=\min\{\ell(w):w\in \mathcal S_n \text{ and } w\geq A\}.  
\end{equation}
\begin{example}
In general ${\tt deg}(A)\neq |D(A)|$.  For example, suppose $A$ is the ASM whose diagram is pictured below.
\begin{center}
$D(A)=\begin{tikzpicture}[x=1.3em,y=1.3em,baseline=3.6em]
     \draw[step=1,gray!30,very thin](0,1) grid (4,5);
     \draw[color=black, thick](0,1)rectangle(4,5);
     \filldraw[color=black, fill=gray!30, thick](0,4)rectangle(1,5);
     \filldraw[color=black, fill=gray, thick](1,4)rectangle(2,5);
     \filldraw[color=black, fill=gray, thick](0,3)rectangle(1,4);
     \filldraw[color=black, fill=gray, thick](2,3)rectangle(3,4);
     \filldraw[color=black, fill=gray, thick](1,2)rectangle(2,3);
     \filldraw [black](2.5,4.5)circle(.1);
     \filldraw [black](1.5,3.5)circle(.1);
     \filldraw [color=black,fill=white,thick](2.5,3.5)circle(.1);
     \filldraw [black](3.5,3.5)circle(.1);
     \filldraw [black](.5,2.5)circle(.1);
     \filldraw [color=black,fill=white,thick](1.5,2.5)circle(.1);
     \filldraw [black](2.5,2.5)circle(.1);
     \filldraw [black](1.5,1.5)circle(.1);
     \draw[thick] (2.5,4.5)--(2.5,4);
     \draw[thick] (3,4.5)--(2.5,4.5);
     \draw[thick] (1.5,3.5)--(1.5,3);
     \draw[thick] (2,3.5)--(1.5,3.5);
     \draw[thick] (3.5,3.5)--(3.5,3);
     \draw[thick] (4,3.5)--(3.5,3.5);
     \draw[thick] (.5,2.5)--(.5,2);
     \draw[thick] (1,2.5)--(.5,2.5);
     \draw[thick] (2.5,2.5)--(2.5,2);
     \draw[thick] (3,2.5)--(2.5,2.5);
     \draw[thick] (1.5,1.5)--(1.5,1);
     \draw[thick] (2,1.5)--(1.5,1.5);
     \draw[thick] (4,4.5)--(3,4.5);
     \draw[thick] (4,2.5)--(3,2.5);
     \draw[thick] (3,1.5)--(2,1.5);
     \draw[thick] (4,1.5)--(3,1.5);
     \draw[thick] (3.5,3)--(3.5,2);
     \draw[thick] (.5,2)--(.5,1);
     \draw[thick] (2.5,2)--(2.5,1);
     \draw[thick] (3.5,2)--(3.5,1);
     \end{tikzpicture}$ 
\hspace{2em}
$r_A=\left(\begin{matrix} 0 & 0 & 1 & 1\\ 0&1&1&2\\ 1&1&2&3\\ 1 &2 & 3 & 4\end{matrix}\right)$
\hspace{2em}
$r_{3412}=\left(\begin{matrix} 0 & 0 & 1 & 1\\ 0&0&1&2\\ 1&1&2&3\\ 1 &2 & 3 & 4\end{matrix}\right)$
\end{center}
Since $r_{3412}<r_A$ we have $3412\geq A$.  Therefore \[{\tt deg}(A)\leq \ell(3412)= 4<|D(A)|.\]  By checking all $w\geq A$, the reader may verify ${\tt deg}(A)=4$.\qed
\end{example}

\subsection{Grassmannian and biGrassmannian permutations}

We recall standard facts on permutations.  See \cite{manivel2001symmetric} as a reference.
A {\bf descent} of $w\in \mathcal S_n$ is a position $i$ so that $w(i)>w(i+1)$.  A permutation  is {\bf Grassmannian} if it has a unique descent.  Let $\mathcal G_n$ denote the set of Grassmannian permutations in $\mathcal S_n$.  
If $u\in \mathcal G_n$, write ${\tt des}(u)$ for the position of its descent.  

Let $u\in \mathcal G_n$.    Let  
\[\lambda_i^{(u)}=u({\tt des}(u)-i+1)-({\tt des}(u)-i+1).\]  
Equivalently, $\lambda_i^{(u)}$ is the number of boxes in row ${\tt des}(u)-i+1$ of $D(u)$.
Since $u$ has a unique descent at ${\tt des}(u)$, we have $\lambda_i^{(u)}\geq \lambda_{i+1}^{(u)}$ for all $i=1,\ldots,{\tt des}(u)-1$.  Then
\[\lambda^{(u)}=(\lambda_1^{(u)}, \lambda_2^{(u)}, \ldots, \lambda_{{\tt des}(u)}^{(u)})\]
is a partition with $\ell(\lambda^{(u)})\leq {\tt des}(u)$ and $\lambda_1^{(u)}\leq n-{\tt des}(u)$.

Write $a\times b$ for the partition whose Young diagram has $a$ rows of length $b$.
\begin{lemma}
\label{lemma:grassbiject}
The map $u\mapsto (\lambda^{(u)},{\tt des}(u))$ defines a bijection between $\mathcal G_n$ and pairs $(\lambda,d)$ with \[\lambda\subseteq d\times (n-d)\] i.e. partitions with $\lambda_1\leq n-d$ and $\ell(\lambda)\leq d$.
\end{lemma}
Let $[\lambda,d]_g$ be the Grassmannian with $\lambda^{([\lambda,d]_g)}=\lambda$ and ${\tt des}([\lambda,d]_g)=d$.  If $\lambda=(0)$ is the empty partition, then  let $[\lambda,d]_g={\rm id}$.
Let $(\boldsymbol \lambda,\mathbf d)$ be as in (\ref{eqn:tupledef}). We can always choose $n$ large enough so that $\lambda^{(i)}\subseteq d_i\times (n-d_i)$ for all $i=1,\ldots,k$.  Then let
\begin{equation}
\mathbf u_{\boldsymbol \lambda,\mathbf d}:=([\lambda^{(1)}, d_1]_g,\ldots,[\lambda^{(k)}, d_k]_g)
\end{equation}
and
\begin{equation}
\label{eqn:associatedasm}
A_{\boldsymbol \lambda,\mathbf d}:=\vee\mathbf u_{\boldsymbol \lambda,\mathbf d}\in{\sf ASM}(n).
\end{equation}

 A permutation is {\bf biGrassmannian} if both it, and its inverse, are Grassmannian.  Write $\mathcal B_n$ for the set of biGrassmannian permutations in $\mathcal S_n$.  A permutation is biGrassmannian if and only if its diagram is a rectangle. 
  Elements of $\mathcal B_n$ are naturally labeled by triples of integers $(i,j,r)$ which satisfy the following conditions:
\begin{enumerate}[label=(B\arabic*),ref=(B\arabic*)]
\item \label{item:B1} $1\leq i,j$
\item \label{item:B2} $0\leq r <\min(i,j)$
\item \label{item:B3} $i+j-r\leq n$.
\end{enumerate}
Let $I_k$ denote the $k\times k$ identity matrix.  Then we write 
\begin{equation}
\label{eqn:biGrDef}
[i,j,r]_b:=\left(\begin{array}{c|c|c|c}
I_{r}& &&\\
\hline
&&I_{i-r}&\\
\hline
&I_{j-r}&&\\
\hline
&&&I_{n-i-j+r}\\
\end{array}
\right)
\end{equation}
 for the (unique) biGrassmannian encoded by this triple.    In the case $r=\min(i,j)$, let $[i,j,r]_b$ be the identity permutation.  

There are multiple labeling conventions for biGrassmannians in the literature (see e.g. \cite{lascoux1996treillis}, \cite{reading2002order}, \cite{kobayashi2013more}).  We have chosen ours so the following properties hold.
\begin{lemma}
\label{lemma:biGrassmannianchar}
Let $u=[i,j,r]_b\in \mathcal B_n$.  Then
\begin{enumerate}
\item  $\mathcal Ess(u)=\{(i,j)\}$
\item $r_u(i,j)=r$.
\item ${\tt des}(u)=i$
\item $\lambda^{(u)}=(i-r_u(i,j))\times (j-r_u(i,j))$.
\end{enumerate}
\end{lemma}
Lemma~\ref{lemma:biGrassmannianchar} is immediate from (\ref{eqn:biGrDef}).

\subsection{Prism tableaux}
\label{subsect:prism}

Each partition $\lambda$ has an associated {\bf Young diagram} which consists of left justified boxes with $\lambda_1$ boxes in the bottom row, $\lambda_2$ in the next, and so on.  We will freely identify $\lambda$ with its Young diagram.
A {\bf reverse semistandard tableau} is a filling of $\lambda$ with positive integers so that labels 
\begin{enumerate}[label=(T\arabic*),ref=(T\arabic*)]
\item \label{item:T1} weakly decrease within rows (from left to right) and 
\item \label{item:T2} strictly decrease (from bottom to top) within columns.
\end{enumerate}  We write ${\tt RSSYT}(\lambda,d)$ for the set of reverse semistandard fillings of $\lambda$ which use labels from the set $[d]$.  
\begin{example}
Let $\lambda=(4,4,2,1)$ and $d=7$.  
\ytableausetup{boxsize=1em, aligntableaux=center}
\[\begin{ytableau}
1\\
3&1\\
4&4&2&2\\
6&6&4&3\\
\end{ytableau}\]
The tableau pictured above is an element of ${\tt RSSYT}(\lambda,7)$. \qed
\end{example}
We define \[{\tt AllPrism}(\boldsymbol \lambda,\mathbf d)={\tt RSSYT}(\lambda^{(1)},d_1)\times \ldots \times {\tt RSSYT}(\lambda^{(k)},d_k).\]  An element of ${\tt AllPrism}(\boldsymbol \lambda,\mathbf d)$ is called a {\bf prism tableau}.

For the discussion which follows, it is not enough to merely think of a prism tableau as a tuple of reverse semistandard tableaux.  Rather, we think of each of the component tableaux as having a position in the $\mathbb Z_{> 0}\times \mathbb Z_{> 0}$ grid.  We use matrix coordinates to refer boxes in  the grid; $(i,j)$ indicates the box in the $i$th row (from the top) and $j$th column (from the left) of the grid.  An {\bf antidiagonal} of $\mathbb Z_{> 0}\times \mathbb Z_{> 0}$ consists of the boxes \[\{(i,1),(i-1,2),\ldots, (1,i)\}.\]

We identify the shape of each $\lambda^{(i)}$ with 
\begin{equation}
\label{eqn:tableauposition}
\lambda^{(i)}=\{(a,b):b\leq \lambda^{(i)}_{d_i-a+1}\}\subseteq \mathbb Z_{>0}\times \mathbb Z_{>0}.
\end{equation}
The {\bf prism shape} for $(\boldsymbol \lambda,\mathbf d)$ is obtained by overlaying the $\lambda^{(i)}$'s: 
\begin{equation}
 \mathbb S(\boldsymbol \lambda,\mathbf d):=\bigcup_{i=1}^k \{(a,b):b\leq \lambda^{(i)}_{d_i-a+1}\}.\end{equation}
From this perspective, a  prism tableau  for $(\boldsymbol \lambda,\mathbf d)$ is a filling of $\mathbb S(\boldsymbol \lambda,\mathbf d)$ which assigns a label of color $i$ from the set $\{1,2,\ldots,d_i\}$ to each $(a,b)\in \lambda^{(i)}$ so that labels of color $i$
 weakly decrease along rows from left to right  and strictly decrease along columns from bottom to top.
Such fillings are  in immediate bijection with ${\tt AllPrism}(\boldsymbol \lambda,\mathbf d)$.
  As such, we freely identify these two representations of a prism tableau.

Weight $\mathcal T$ as follows:
\[{\tt wt}(\mathcal T)=\prod_{i=1}^\infty x_i^{n_i}\] where 
$n_i$ is the number of antidiagonals which contain the label $i$ (in any color).

\begin{example}
\label{example:prism}
Let $\boldsymbol \lambda=((1),(3,2),(2,1,1))$ and $\mathbf d=(2,5,6)$.  Below, we give an example of $\mathcal T\in {\tt AllPrism}(\boldsymbol \lambda,\mathbf d)$.
\ytableausetup{boxsize=.95em, aligntableaux=center}
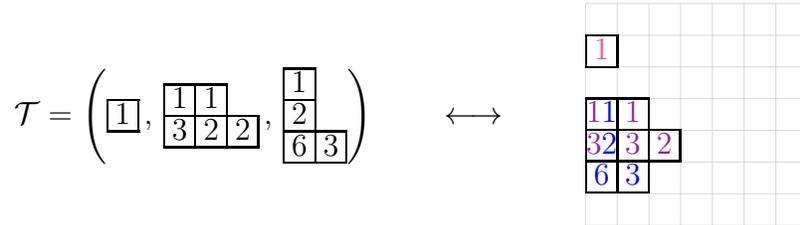
\[
\mathcal T=\left(
\begin{ytableau}
1
\end{ytableau}
\, , \,
\begin{ytableau}
1&1\\
3&2&2
\end{ytableau}
\, , \, 
\begin{ytableau}
1\\
2\\
6&3
\end{ytableau}
 \right) 
\hspace{2em}
\longleftrightarrow
\hspace{2em}
\begin{tikzpicture}[x=1em,y=1em,baseline=4.2em]
\draw[step=1,gray!30,very thin] (0,1) grid (7,8); 
\node[] at (1.5,5){
\begin{ytableau}
\none\\
{\color{Rhodamine}1}\\
\none\\
{\color{Plum}1}{\color{blue}1}&{\color{Plum}1}\\
{\color{Plum}3}{\color{blue}2}&{\color{Plum}3}&{\color{Plum}2}\\
{\color{blue}6}&{\color{blue}3}\\
\end{ytableau}};
\end{tikzpicture}
\]
The corresponding weight monomial is ${\tt wt}(\mathcal T)=x_1^3x_2^2x_3^3x_6$.
\qed
\end{example}

Let 
\begin{equation}
\label{eqn:deglambdad}
{\tt deg}(\boldsymbol \lambda,\mathbf d)=\min\{{\tt deg}({\tt wt}(\mathcal T)):\mathcal T\in {\tt AllPrism}(\boldsymbol \lambda,\mathbf d)\}.
\end{equation}
$\mathcal T\in {\tt AllPrism}(\boldsymbol \lambda,\mathbf d)$ is {\bf minimal} if ${\tt deg}({\tt wt}(\mathcal T))={\tt deg}(\boldsymbol \lambda,\mathbf d).$
 Let $\ell_c$ be a label $\ell$ of color $c$. 
Labels $\{\ell_c,\ell_d, \ell'_e\}$
in the same antidiagonal form an {\bf unstable triple}
if $\ell<\ell'$ and replacing the $\ell_c$ with $\ell'_c$
gives a prism tableau.
Write
\begin{equation}{\tt Prism}(\boldsymbol \lambda,\mathbf d)=\{\mathcal T\in {\tt AllPrism}(\boldsymbol \lambda,\mathbf d): \mathcal T \text{ is minimal and has no unstable triples}\}.
\end{equation}

We now describe two ways of taking an ASM as a input and producing a pair $(\boldsymbol \lambda,\mathbf d)$ so that $A=A_{\boldsymbol \lambda,\mathbf d}$.  Both procedures are entirely combinatorial.  We start with BiGrassmannian prism tableaux, which were defined in \cite{weigandt2015prism}.

\begin{definition}[BiGrassmannian Prism Tableaux]
Suppose \[{\mathcal Ess}(A)=\{(i_1,j_1),(i_2,j_2),\ldots,(i_k,j_k)\}.\]  Let
\begin{equation}
\label{eqn:betapartition}
\beta^{(\ell)}=(i_\ell-r_A(i_\ell,j_\ell))\times(j_\ell-r_A(i_\ell,j_\ell)).
\end{equation}  Define
 $\boldsymbol \beta_A=(\beta^{(1)},\ldots,\beta^{(k)})$ and $\mathbf b_A=\{i_1,\ldots,i_k\}.$
The {\bf biGrassmannian prism shape} is  $\mathbb S_B(A):=\mathbb S(\boldsymbol \beta_A,\mathbf b_A)$.
Write ${\tt Prism}_B(A):={\tt Prism}(\boldsymbol \beta_A,\mathbf b_A)$.
\end{definition}

\begin{example}
Let $A$ be as in Example~\ref{example:asmDiagram}. Then $\mathcal Ess(A)=\{(1,3),(2,1),(3,2)\}$. 
\begin{center}
\begin{tabular}{|c|c|c|} \hline
$(i_\ell,j_\ell)$ & $r_A(i_\ell,j_\ell)$ & $\beta^{(\ell)}$ \\ \hline
$(1,3)$ & 0  &  $1\times 3$  \\ \hline
$(2,1)$ & 0 &  $2\times 1$ \\ \hline
$(3,2)$ & 1 &  $2\times 1$  \\ \hline
\end{tabular}
\end{center}
Using the table above, we construct the shape $\mathbb S_B(A)$.

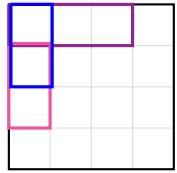
\[\begin{tikzpicture}[x=1.3em,y=1.3em,baseline=3.6em]
\draw[step=1,gray!30,very thin] (0,1) grid (4,5); 
\draw[color=black, thick](0,1)rectangle(4,5);
\draw[color=Plum,very thick] (0,5)--(3,5)--(3,4)--(0,4)--(0,5);
\draw[color=Rhodamine, very thick] (0,4.05)--(1,4.05)--(1,2)--(0,2)--(0,4);
\draw[color=blue,very thick ] (0.05,5)--(1.05,5)--(1.05,3)--(0.05,3)--(0.05,5);
     \end{tikzpicture}\]
There are only three prism fillings of $\mathbb S_B(A)$.
\ytableausetup{boxsize=.95em, aligntableaux=top,nobaseline}
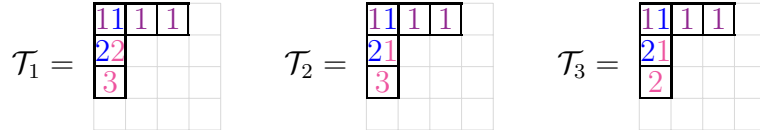
\[
\mathcal T_1=
\begin{tikzpicture}[x=1em,y=1em,baseline=2.8em]
\draw[step=1,gray!30,very thin] (0,1) grid (4,5); 
\node[] at (1.5,3.5){
\begin{ytableau}
{\color{Plum}1}{\color{blue}1}&{\color{Plum}1}&{\color{Plum}1}\\
{\color{blue}2}{\color{Rhodamine}2}\\
{\color{Rhodamine}3}\\
\end{ytableau}};
\end{tikzpicture}
\hspace{2em}
\mathcal T_2=
\begin{tikzpicture}[x=1em,y=1em,baseline=2.8em]
\draw[step=1,gray!30,very thin] (0,1) grid (4,5); 
\node[] at (1.5,3.5){
\begin{ytableau}
{\color{Plum}1}{\color{blue}1}&{\color{Plum}1}&{\color{Plum}1}\\
{\color{blue}2}{\color{Rhodamine}1}\\
{\color{Rhodamine}3}\\
\end{ytableau}};
\end{tikzpicture}
\hspace{2em}
\mathcal T_3=
\begin{tikzpicture}[x=1em,y=1em,baseline=2.8em]
\draw[step=1,gray!30,very thin] (0,1) grid (4,5); 
\node[] at (1.5,3.5){
\begin{ytableau}
{\color{Plum}1}{\color{blue}1}&{\color{Plum}1}&{\color{Plum}1}\\
{\color{blue}2}{\color{Rhodamine}1}\\
{\color{Rhodamine}2}\\
\end{ytableau}};
\end{tikzpicture}
\]
The corresponding weight monomials are ${\tt wt}(\mathcal T_1)=x_1^3x_2x_3$, ${\tt wt}(\mathcal T_2)=x_1^3x_2x_3$, and ${\tt wt}(\mathcal T_3)=x_1^3x_2^2$.  These all have the same degree, and so each tableaux is minimal.  $\mathcal T_1$ can be obtained from $\mathcal T_2$ by replacing the pink 1 with a 2.  So $\mathcal T_2$ has an unstable triple.  So we conclude 
\[{\tt Prism}_B(A)=\{T_1,T_3\}.\] 
Then $\mathfrak A_{\boldsymbol \beta_A,\mathbf b_A}=x_1^3x_2x_3+x_1^3x_2^2$.
\qed
\end{example}

We now introduce the parabolic prism model.  Our definition uses the \emph{monotone triangles} of W.~H.~Mills, D.~P.~Robbins, and H.~Rumsey \cite{mills1983alternating}.  
Given \[A=(a_{i j})_{i,j=1}^n\in{\sf ASM}(n)\] let $C_A$ be the matrix of partial column sums, i.e. $C_A(i,j)=\sum_{\ell=1}^i a_{\ell j}$. The $i$th row of $m_A$ records (in increasing order) the positions of the 1s in the $i$th row of $C_A$.  The array $m_A$ is called a {\bf monotone triangle}. 
\begin{example}
\label{example:monotonetriangle}
\[A=
\begin{pmatrix}
0& 0&1& 0\\
1&0&-1&1\\
0&0&1&0\\
0&1&0&0
\end{pmatrix}
\hspace{2em}
C_A=\begin{pmatrix}
0& 0&1& 0\\
1&0&0&1\\
1&0&1&1\\
1&1&1&1
\end{pmatrix}
\hspace{2em}
m_A=
\begin{array}{cccccccccc}
&&&&3\\
&&&1&&4\\
&&1&&3&&4\\
&1&&2&&3&&4
\end{array}
\]
\qed
\end{example}
There is explicit dictionary between monotone triangles and corner sum matrices.  Entry $(i,j)$ of $m_A$ indicates the position of the $j$th ascent in row $i$ of $r_A$, i.e, 
\begin{equation}
\label{eqn:trianglerank}
m_A(i,j)=a \text{ if and only if } r_A(i,a-1)=j-1 \text{ and } r_A(i,a)=j.
\end{equation}

Given $A$ and $1\leq \ell\leq n$, we define 
\begin{equation}
\lambda^{(A,\ell)}=(m_A(\ell,\ell)-\ell,m_A(\ell,\ell-1)-(\ell-1),\ldots, m_A(\ell,1)-1).
\end{equation}
Since $m_A$ strictly increases along rows, $\lambda^{(A,\ell)}$ is a partition.  By construction, \[\lambda^{(A,\ell)}\subseteq \ell\times (n-\ell).\]
Notice if $u\in \mathcal G_n$, then $\lambda^{(u,{\tt des}(u))}=\lambda^{(u)}$.
If $w\in \mathcal S_n$, then $[\lambda^{(A,\ell)},\ell]_g$ is the minimal length coset representative for $w$ in the maximal parabolic subgroup generated by removing $(\ell,\ell+1)$ from the list of simple transpositions in $\mathcal S_n$.

\begin{definition}[Parabolic Prism Tableaux]
 Write \[\{i:(i,j)\in \mathcal Ess(A)\}=\{i_1,\ldots,i_k\}\] for the indices of essential rows of $A$.  Let \[\text{$\boldsymbol \rho_A=(\lambda^{(A,i_1)}, \lambda^{(A,i_2)}, \ldots, \lambda^{(A,i_k)})$ and $\mathbf p_A=(i_1,\ldots,i_k)$}.\]  Then define the {\bf parabolic prism shape} 
\[\mathbb S_P(A)=\mathbb S(\boldsymbol \rho_A, \mathbf p_A).\]
We abbreviate ${\tt Prism}_P(A):={\tt Prism}(\boldsymbol \rho_A,\mathbf p_A)$.
\end{definition}

\begin{example}
Let $A$ be as in Example~\ref{example:asmDiagram}.  Then
\[m_A=
\begin{array}{cccccccccc}
&&&&4\\
&&&2&&4\\
&&1&&3&&4\\
&1&&2&&3&&4
\end{array}.\]
The essential rows are $\mathbf p_A=(1,2,3)$ and $\boldsymbol \rho_A=((3),(2,1),(1,1))$.
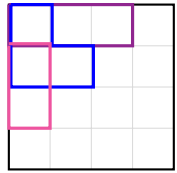
\[\begin{tikzpicture}[x=1.3em,y=1.3em,baseline=3.6em]
\draw[step=1,gray!30,very thin] (0,1) grid (4,5); 
\draw[color=black, thick](0,1)rectangle(4,5);
\draw[color=Plum,very thick] (0,5)--(3,5)--(3,4)--(0,4)--(0,5);
\draw[color=blue,very thick ] (0.05,5)--(1.05,5)--(1.05,4)--(2.05,4)--(2.05,3)--(0.05,3)--(0.05,5);
\draw[color=Rhodamine, very thick] (0,4.05)--(1,4.05)--(1,2)--(0,2)--(0,4);
     \end{tikzpicture}\]
We list the possible prism fillings of $(\boldsymbol \rho_A, \mathbf p_A)$.
\ytableausetup{boxsize=.95em, aligntableaux=top}
\[
\mathcal T_1=
\begin{tikzpicture}[x=1em,y=1em,baseline=2.8em]
\draw[step=1,gray!30,very thin] (0,1) grid (4,5); 
\node[] at (1.5,3.5){
\begin{ytableau}
{\color{Plum}1}{\color{blue}1}&{\color{Plum}1}&{\color{Plum}1}\\
{\color{blue}2}{\color{Rhodamine}2}&{\color{blue}2}\\
{\color{Rhodamine}3}\\
\end{ytableau}};
\end{tikzpicture}
\hspace{1em}
\mathcal T_2=
\begin{tikzpicture}[x=1em,y=1em,baseline=2.8em]
\draw[step=1,gray!30,very thin] (0,1) grid (4,5); 
\node[] at (1.5,3.5){
\begin{ytableau}
{\color{Plum}1}{\color{blue}1}&{\color{Plum}1}&{\color{Plum}1}\\
{\color{blue}2}{\color{Rhodamine}1}&{\color{blue}2}\\
{\color{Rhodamine}3}\\
\end{ytableau}};
\end{tikzpicture}
\hspace{1em}
\mathcal T_3=
\begin{tikzpicture}[x=1em,y=1em,baseline=2.8em]
\draw[step=1,gray!30,very thin] (0,1) grid (4,5); 
\node[] at (1.5,3.5){
\begin{ytableau}
{\color{Plum}1}{\color{blue}1}&{\color{Plum}1}&{\color{Plum}1}\\
{\color{blue}2}{\color{Rhodamine}1}&{\color{blue}2}\\
{\color{Rhodamine}2}\\
\end{ytableau}};
\end{tikzpicture}
\hspace{1em}
\mathcal T_4=
\begin{tikzpicture}[x=1em,y=1em,baseline=2.8em]
\draw[step=1,gray!30,very thin] (0,1) grid (4,5); 
\node[] at (1.5,3.5){
\begin{ytableau}
{\color{Plum}1}{\color{blue}1}&{\color{Plum}1}&{\color{Plum}1}\\
{\color{blue}2}{\color{Rhodamine}2}&{\color{blue}1}\\
{\color{Rhodamine}3}\\
\end{ytableau}};
\end{tikzpicture}
\]
\[
\mathcal T_5=
\begin{tikzpicture}[x=1em,y=1em,baseline=2.8em]
\draw[step=1,gray!30,very thin] (0,1) grid (4,5); 
\node[] at (1.5,3.5){
\begin{ytableau}
{\color{Plum}1}{\color{blue}1}&{\color{Plum}1}&{\color{Plum}1}\\
{\color{blue}2}{\color{Rhodamine}1}&{\color{blue}1}\\
{\color{Rhodamine}3}\\
\end{ytableau}};
\end{tikzpicture}
\hspace{1em}
\mathcal T_6=
\begin{tikzpicture}[x=1em,y=1em,baseline=2.8em]
\draw[step=1,gray!30,very thin] (0,1) grid (4,5); 
\node[] at (1.5,3.5){
\begin{ytableau}
{\color{Plum}1}{\color{blue}1}&{\color{Plum}1}&{\color{Plum}1}\\
{\color{blue}2}{\color{Rhodamine}1}&{\color{blue}1}\\
{\color{Rhodamine}2}\\
\end{ytableau}};
\end{tikzpicture}
\]
\begin{center}
\begin{tabular}{|c|c|c|c|c|c|c|} \hline
$i$& 1&2&3&4&5&6\\ \hline
${\tt wt}(\mathcal T_i)$ &
 $x_1^3x_2^2x_3$&
$x_1^3x_2^2x_3$&
$x_1^3x_2^2$&
$x_1^3x_2x_3$&
$x_1^3x_2x_3$&
$x_1^3x_2^2$\\ \hline
minimal& no&no&yes&yes&yes&yes\\ \hline
\end{tabular}
\end{center}
Among the minimal tableaux, $\mathcal T_4$ is obtained by replacing the unstable triple in $\mathcal T_5$.   Likewise, replacing the unstable triple $\mathcal T_6$ produces $\mathcal T_3$.  So ${\tt Prism}_P(A)=\{\mathcal T_3,\mathcal T_4\}$.
Then 
\[
\mathfrak A_{\boldsymbol \rho_A,\mathbf p_A}=x_1^3x_2^2+x_1^3x_2x_3. 
\]
Notice that $\mathfrak A_{\boldsymbol \beta_A,\mathbf b_A}=\mathfrak A_{\boldsymbol \rho_A,\mathbf p_A}$.  This holds in general as a consequence of Theorem~\ref{theorem:main} and the next proposition. \qed
\end{example}

\begin{proposition}
\label{prop:therightasm}
\begin{enumerate}
\item $A=A_{\boldsymbol \beta_A,\mathbf b_A}$.
\item $A=A_{\boldsymbol \rho_A,\mathbf p_A}$.
\end{enumerate}
\end{proposition}
We will postpone the proof to Section~\ref{subsection:corner}.

\section{The Lattice of ASMs}
\label{section:asmlattice}

\subsection{Preliminaries on posets and lattices}

  We follow  \cite{lascoux1996treillis} and \cite{reading2002order} as references.  
  A {\bf partially ordered set} (or poset) is a set $\mathcal P$ equipped with a binary relation $\leq$ which satisfies the axioms of \emph{reflexivity}, \emph{antisymmetry}, and \emph{transitivity}.         
 If $a\leq b$ and $a\neq b$ we write $a<b$.
Given $a,b\in \mathcal P$ we say $b$ {\bf covers} $a$ if $a< b$ and whenever $a\leq c\leq b$, we have $c=a$ or $c=b$.
 
An element $a\in \mathcal P$ is {\bf minimal} in $\mathcal P$ if whenever $b\in \mathcal P$ so that $b\leq a$ we have $a=b$.  
Similarly, $a\in \mathcal P$ is {\bf maximal} in $\mathcal P$ if whenever $b\in \mathcal P$ so that $b\geq a$ we have $a=b$.  Write ${\tt MIN}(\mathcal P)$ for the set of minimal elements in $\mathcal P$ and  ${\tt MAX}(\mathcal P)$ for the maximal elements.

The {\bf join} of $\mathcal S\subseteq\mathcal P$ (when it exists) is the least upper bound of $\mathcal S$.  Similarly, the {\bf meet} is the greatest lower bound. The join and meet are denoted $\vee $ and $\wedge $ respectively.

An element $a\in \mathcal P$ is {\bf basic} if $a\neq \vee \mathcal S$ whenever $a\not \in \mathcal S$.  The set of basic elements in $\mathcal P$ is called the {\bf base} of $\mathcal P$.  Let $\mathbb P(\mathcal S)$ denote the {\bf power set} of $\mathcal S$, that is the set of all subsets of $\mathcal S$.  $\mathbb P(\mathcal S)$ has the natural structure of a poset by inclusion of sets.  Given any subset $\mathcal C\subseteq \mathcal P$, define $\pi_{\mathcal C}:\mathcal P\rightarrow \mathbb P(\mathcal C)$ by $\pi_{\mathcal C}(a)=\{c\in \mathcal C:c\leq a\}$.
The base is characterized by the following property.
\begin{proposition}[{\cite[Proposition 2.4]{lascoux1996treillis}}]
\label{prop:base}
Let $\mathcal B$ be the base of a finite poset $\mathcal P$.  The projection $\pi_{\mathcal B}$ is an order isomorphism onto its image.  Furthermore, if any $\mathcal C\subseteq \mathcal P$ has this property, then $\mathcal B\subseteq\mathcal  C$.
\end{proposition}
As a consequence, any element  $a\in\mathcal P$ is uniquely encoded by the set $\pi_{\mathcal B}(a)$.  Furthermore, $a=\vee\pi_{\mathcal B}(a)$ (see \cite[Proposition 9]{reading2002order}).  In particular, $a=\vee{\tt MAX}(\pi_{\mathcal B}(a))$.

 A {\bf lattice} $\mathcal L$ is a poset in which every pair of elements has a join and a meet.  Basic elements in a lattice are also known as {\bf join-irreducibles} and have the characterization that they cover a unique element.  
 A {\bf sublattice} of $\mathcal L$ is a subset $\mathcal L'\subseteq \mathcal L$ which is itself a lattice and has the \emph{same} operations of join and meet as $\mathcal L$.

 Assume $\mathcal S$ is a totally ordered set and $I$ some indexing set.   Let 
\[\mathcal S^I:=\{(a_i)_{i\in I}:a_i\in \mathcal S \text{ for all } i\in I\}.\]  
There is a natural partial order on $\mathcal S^I$ by entrywise comparison.  Explicitly, if $\mathbf a=(a_i)_{i\in I}$ and $\mathbf b=(b_i)_{i\in I}$ in $\mathcal S^I$, then 
\[\mathbf a\leq  \mathbf b \text{ if and only if } a_i\leq b_i \text{ for all } i\in I.\]  Write
\[\boldsymbol \max(\mathbf a,\mathbf b):=(\max(a_i,b_i))_{i\in I} \quad \text{ and } \quad \boldsymbol\min(\mathbf a,\mathbf b):=(\min(a_i,b_i))_{i\in I}.\]
\begin{lemma}
\label{lemma:latticesituation}
\begin{enumerate}
\item \label{item:latticesituation1} $\mathcal S^I$ is a lattice with $\mathbf a\vee \mathbf b=\boldsymbol\max(\mathbf a,\mathbf b)$ and $\mathbf a\wedge \mathbf b=\boldsymbol\min(\mathbf a,\mathbf b)$.
\item  \label{item:latticesituation2} If a subset of $\mathcal S^I$ is closed under joins and meets, then it is a sublattice of $\mathcal S^I$ (and hence is itself a lattice).
\end{enumerate}
\end{lemma}
Since $\mathcal S^I$ is a cartesian product of lattices, (\ref{item:latticesituation1}) is immediate.  Likewise, (\ref{item:latticesituation2}) follows from the definition of a sublattice.

A lattice is {\bf complete} if every subset has a join and meet.  Any finite lattice is automatically complete. The {\bf Dedekind-MacNeille} completion of $\mathcal P$ is the smallest complete lattice which contains $\mathcal P$ as an order embedding.  Any finite poset has the same base as its  Dedekind-MacNeille  completion \cite[Proposition 28]{reading2002order}.

\subsection{The  Dedekind-MacNeille completion of the symmetric group}

Write \[{\sf R}(n):=\{r_A:A\in{\sf ASM}(n)\}.\]    For convenience, define $r_A(i,j)=0$ whenever $i=0$ or $j=0$.   Then
\begin{equation}
a_{ij}=r_A(i,j)-r_A(i,j-1)-r_A(i-1,j)+r_A(i-1,j-1)
\end{equation}
 recovers the $(i,j)$ entry of $A$  \cite{robbins1986determinants}.   As such, the map $A\mapsto r_A$ defines a bijection between ${\sf ASM}(n)$ and ${\sf R}(n)$.
 The following lemma characterizes corner sums of ASMs.
\begin{lemma}[{\cite[Lemma 1]{robbins1986determinants}}]
Let $A$ be an $n\times n$ matrix.  Then $A\in {\sf ASM}(n)$ if and only if
\begin{enumerate}[label=(R\arabic*),ref=(R\arabic*)]
\item \label{item:R1} $r_A(i,n)=r_A(n,i)=i$ for all $i=1,\ldots,n$.
\item \label{item:R2} $r_A(i,j)-r_A(i-1,j) \text{ and } r_A(i,j)-r_A(i,j-1)\in \{0,1\}$ for all $1\leq i,j\leq n$.
\end{enumerate}
\end{lemma}

\begin{lemma}
\label{lemma:rankbruhat}
${\sf R}(n)$ is a distributive lattice with  join and  meet given by $r_A\vee r_B=\boldsymbol\max(r_A,r_B)$ and    $r_A\wedge r_B=\boldsymbol\min(r_A,r_B)$, respectively.
\end{lemma}
Lemma~\ref{lemma:rankbruhat} follows from Lemma~\ref{lemma:latticesituation} by verifying that \ref{item:R1} and \ref{item:R2} are preserved under taking minimums and maximums.  The lattice of ASMs was initially studied by N.~Elkies, G.~Kuperberg, M.~Larsen, and J.~Propp  \cite{elkies1992alternating}.  The definition in \emph{ibid} is in terms of height functions, which are in  obvious order reversing bijection with corner sum matrices. 
The order on ${\sf ASM}(n)$ can also be defined using monotone triangles; this perspective was used in \cite{lascoux1996treillis}.
\begin{lemma}[{\cite[Lemma 5.4]{lascoux1996treillis}}]
\label{lemma:symgroupbase}
The  Dedekind-MacNeille completion of $\mathcal S_n$ is isomorphic to ${\sf ASM}(n)$. 
The base of the $\mathcal S_n$, and hence ${\sf ASM}(n)$, is $\mathcal B_n$.
\end{lemma}
   In \cite{lascoux1996treillis}, A.~Lascoux and M.~P.~Sch\"utzenberger also give the base for type $B$ Coxeter groups.  M.~Geck and S.~Kim determined the base for all finite Coxeter groups \cite{geck1997bases}.

 Let
\begin{equation}
\label{eqn:bigrass}
{\tt biGr}(A)={\tt MAX}( \pi_{\mathcal B_n}(A))
\end{equation}
be the maximal biGrassmannians in $\pi_{\mathcal B_n}(A)$.
Then as a consequence of Lemma~\ref{lemma:symgroupbase} and Proposition~\ref{prop:base}
\begin{equation}
\label{eqn:bigrjoin}
A=\vee \pi_{\mathcal B_n}(A)=\vee{\tt biGr}(A).
\end{equation}
Notice that (\ref{eqn:bigrjoin}) determines a pair $(\boldsymbol \lambda,\mathbf d)$ so that $A=A_{\boldsymbol \lambda,\mathbf d}$.  If ${\tt biGr}(A)=\{u_1,\ldots,u_k\}$ then setting $\lambda^{(i)}=\lambda^{(u_i)}$ and $d_i={\tt des}(u_i)$ produces the desired prism shape.

We also define
\begin{equation}
\label{eqn:PermA}
{\tt Perm}(A):={\tt MIN}(\{w\in \mathcal S_n:w\geq A\})\end{equation}
and 
\begin{equation}
\label{example:permminperm}
{\tt MinPerm}(A):=\{w\in {\tt Perm}(A):\ell(w)={\tt deg}(A)\}.
\end{equation}
\begin{example}
\label{example:noneqi}
Let $A$ be the ASM whose diagram is pictured below.
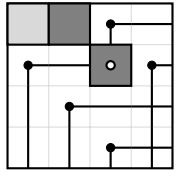
\[
\begin{tikzpicture}[x=1.3em,y=1.3em]
     \draw[step=1,gray!30,very thin](0,1) grid (4,5);
     \draw[color=black, thick](0,1)rectangle(4,5);
     \filldraw[color=black, fill=gray!30, thick](0,4)rectangle(1,5);
     \filldraw[color=black, fill=gray, thick](1,4)rectangle(2,5);
     \filldraw[color=black, fill=gray, thick](2,3)rectangle(3,4);
     \filldraw [black](2.5,4.5)circle(.1);
     \filldraw [black](.5,3.5)circle(.1);
     \filldraw [color=black,fill=white,thick](2.5,3.5)circle(.1);
     \filldraw [black](3.5,3.5)circle(.1);
     \filldraw [black](1.5,2.5)circle(.1);
     \filldraw [black](2.5,1.5)circle(.1);
     \draw[thick] (2.5,4.5)--(2.5,4);
     \draw[thick] (3,4.5)--(2.5,4.5);
     \draw[thick] (.5,3.5)--(.5,3);
     \draw[thick] (1,3.5)--(.5,3.5);
     \draw[thick] (3.5,3.5)--(3.5,3);
     \draw[thick] (4,3.5)--(3.5,3.5);
     \draw[thick] (1.5,2.5)--(1.5,2);
     \draw[thick] (2,2.5)--(1.5,2.5);
     \draw[thick] (2.5,1.5)--(2.5,1);
     \draw[thick] (3,1.5)--(2.5,1.5);
     \draw[thick] (4,4.5)--(3,4.5);
     \draw[thick] (2,3.5)--(1,3.5);
     \draw[thick] (3,2.5)--(2,2.5);
     \draw[thick] (4,2.5)--(3,2.5);
     \draw[thick] (4,1.5)--(3,1.5);
     \draw[thick] (.5,3)--(.5,2);
     \draw[thick] (3.5,3)--(3.5,2);
     \draw[thick] (.5,2)--(.5,1);
     \draw[thick] (1.5,2)--(1.5,1);
     \draw[thick] (3.5,2)--(3.5,1);
     \end{tikzpicture}
\]
By direct verification, we may compute ${\tt Perm}(A)=\{3412,4123\}$.   Since $\ell(3412)=4$ and $\ell(4123)=3$, we have ${\tt MinPerm}(A)=\{4123\}$. \qed
\end{example}

\subsection{Corner sums and biGrassmannians}
\label{subsection:corner}
In this section, we discuss the specific connection of ${\tt biGr}(A)$ to $\mathcal Ess(A)$.  Furthermore, we review known facts about biGrassmannian permutations and the Bruhat order.  We then use this to prove Proposition~\ref{prop:therightasm}.

The definition of $\mathcal Ess(A)$ generalizes W.~Fulton's definition of the essential set of a permutation matrix.  However, there is another characterization in terms of corner sum matrices.  This is taken as the definition elsewhere in the literature, for example see \cite{fortin2008macneille} or \cite{kobayashi2013more}.  We prove these definitions are equivalent.
\begin{lemma}
\label{lemma:essentialrankchar}
$\mathcal Ess(A)=\{(i,j):r_A(i,j)=r_A(i-1,j)=r_A(i,j-1) \text{ and } r_A(i,j)+1=r_A(i+1,j)=r_A(i,j+1)\}$.
\end{lemma}

\begin{proof}
By (\ref{eqn:inversion2}), if $(i,j)\in D(A)$ if and only if $\sum_{k=1}^ia_{kj}=0$ and $\sum_{l=1}^ja_{il}=0$.
Since \[r_A(i,j)-r_A(i-1,j)=\sum_{l=1}^ja_{il} \quad  \text{ and } \quad r_A(i,j)-r_A(i,j-1)=\sum_{k=1}^ia_{kj}\] 
we have
\begin{equation}
\label{eqn:diagramchar}
(i,j)\in D(A) \quad \text{ if and only if } \quad r_A(i,j)=r_A(i-1,j)=r_A(i,j-1).
\end{equation}

\noindent $(\subseteq)$
Assume $(i,j)\in \mathcal Ess(A)$.  By definition, $(i+1,j),(i,j+1)\not\in D(A)$.  Since $(i,j)\in \mathcal D(A)$, applying (\ref{eqn:diagramchar}) and \ref{item:R2}, we have \[r_A(i,j)=r_A(i,j-1)\leq r_A(i+1,j-1)\leq r_A(i+1,j).\] If $r_A(i+1,j)=r_A(i,j)$ then $r_A(i+1,j)=r_A(i+1,j-1)$.  Then by (\ref{eqn:diagramchar}), we have $(i+1,j-1)\in \mathcal D(A)$, contradicting $(i,j)\in \mathcal Ess(A)$.  So $r_A(i,j)+1=r_A(i+1,j)$.  The argument for $r_A(i,j)+1=r_A(i,j+1)$ is entirely analogous.

\noindent $(\supseteq)$
By assumption, $r_A(i,j)=r_A(i-1,j)=r_A(i,j-1)$.  Then applying (\ref{eqn:diagramchar}), $(i,j)\in \mathcal D(A)$.  Since $r_A(i,j)\neq r_A(i+1,j)$ and $r_A(i,j)\neq r_A(i,j+1)$, we conclude 
\[(i+1,j),(i,j+1)\not \in \mathcal D(A).\]  So $(i,j)\in \mathcal Ess(A)$.
\end{proof}

\begin{lemma}
\label{lemma:bigrassranktest}
$[i,j,r]_b=\wedge\{A\in{\sf ASM}(n):r_A(i,j)\leq r\}$.
\end{lemma}
See {\cite[Theorem 30]{brualdi2017alternating}} for a proof.  An analogous statement in terms of monotone triangles appears in \cite{lascoux1996treillis}.  Note in particular, 
\begin{equation}
\text{if $r_A(i,j)\leq r$, we have $[i,j,r]_b\leq A$.}
\end{equation}
  This is a special case of the generalized essential criterion given in \cite{kobayashi2013more}.

\begin{lemma}
\label{lemma:bigrallentries}
Fix $A\in {\sf ASM}(n)$. 
\begin{enumerate}
\item For all $1\leq i,j\leq n$, $[i,j,r_A(i,j)]_b\in \mathcal B_n$ or $[i,j,r_A(i,j)]_b={\rm id}$.
\item $A=\vee\{[i,j,r_A(i,j)]_b:1\leq i,j\leq n\}$.
\end{enumerate}
\end{lemma}
\begin{proof}

\noindent (1)  From \ref{item:R2} we must have $r_A(i,j)\leq \min\{i,j\}$.  If this is an equality, we have $[i,j,r_A(i,j)]_b={\rm id}$ and we are done.  So assume not.  By \ref{item:R1}, $r_A(i,n)=i$.  As a consequence of \ref{item:R2}, $n-j\geq i-r_A(i,j)$.  Then $i+j-r_A(i,j)\leq n$.  So the conditions \ref{item:B1}-\ref{item:B3} are satisfied.

\noindent (2)  Let $A'= \vee\{[i,j,r_A(i,j)]_b:1\leq i,j\leq n]\}$. By Lemma~\ref{lemma:bigrassranktest}, $A$ is an upper bound to each $[i,j,r_A(i,j)]_b$.  So $A\geq A'$ and $r_A\leq  r_{A'}$.  Since $r_{A'}$ is entrywise the minimum of the corner sum matrices of the $[i,j,r_A(i,j)]_b$'s, in particular, $r_{A'}(i,j)\leq r_{[i,j,r_A(i,j)]_b}(i,j)=r_A(i,j)$.  Then $r_{A'}\leq r_A$.  As such, $r_{A'}=r_A$ and so $A'=A$.
\end{proof}

\begin{lemma}
\label{lemma:throwawaystuff}
Assume $A\neq I_n$.
If $(i,j)\not \in \mathcal Ess(A)$, then there is some $(i',j')$ so that \[[i,j,r_A(i,j)]_b<[i',j',r_A(i',j')]_b.\]
\end{lemma}
\begin{proof}

Let $u=[i,j,r_A(i,j)]_b$.  If $r_A(i,j)=\min\{i,j\}$ then $u$ is the identity and hence smaller than any  biGrassmannian.  So assume $r_A(i,j)<\min\{i,j\}$. Since $A\neq I_n$, there is some $(i',j')$ exists for which $[i',j',r_A(i',j')]_b\in \mathcal B_n$  (e.g. some $(i',j')\in \mathcal Ess(A)$).  

Applying Lemma~\ref{lemma:essentialrankchar} and \ref{item:R2}, there are four potential ways for $(i,j)$ to fail to be in $\mathcal Ess(A)$.

\noindent Case 1: $r_A(i,j)= r_A(i-1,j)+1$.

Since we have assumed $r_A(i,j)<\min\{i,j\}$ and $r_A(0,j)=0$, we must have $i>1$.
So let $u'=[i-1,j,r_A(i-1,j)]_b$.  

Then $r_{u'}(i,j)=r_{u'}(i-1,j)+1=r_A(i-1,j)+1=r_A(i,j)=r_u(i,j)$ so by Lemma~\ref{lemma:bigrassranktest}, $u\leq u'$.

\noindent Case 2: $r_A(i,j)=r_A(i,j-1)+1$.

The argument is entirely analogous to Case 1.

\noindent Case 3: $r_A(i,j)= r_A(i+1,j)$.

Now let $u'=[i+1,j,r_A(i+1,j)]_b$.  Then \[r_{u'}(i,j)=r_{u'}(i+1,j)=r_A(i+1,j)=r_A(i,j)=r_u(i,j).\]  Applying 
Lemma~\ref{lemma:bigrassranktest}, we have $u<u'$.

\noindent Case 4: $r_A(i,j)= r_A(i,j+1)$.

This is essentially the same as Case 3.
\end{proof}

The following proposition shows how to recover ${\tt biGr}(A)$ from $\mathcal Ess(A)$. 
\begin{proposition}
\label{prop:bigrassList}
${\tt biGr}(A)=\{[i,j,r_A(i,j)]_b:(i,j)\in \mathcal Ess(A)\}.$
\end{proposition}
Proposition~\ref{prop:bigrassList} is  discussed in \cite[Section~5]{lascoux1996treillis},  using essential points of  monotone triangles.  It can be found in a slightly more general context in \cite[Theorem~5.1]{fortin2008macneille}.     As an immediate consequence,  $A$ is determined by the restriction of $r_A$ to $\mathcal Ess(A)$.  This generalizes \cite[Lemma 3.10]{fulton1992flags}.
\begin{proof}[Proof of Proposition~\ref{prop:bigrassList}]

First note that 
\begin{equation}
\label{eqn:biGrcontain}
{\tt biGr}(A)\subseteq \{[i,j,r_A(i,j)]_b:(i,j)\in \mathcal Ess(A)\}.
\end{equation}  If $A=I_n$  then ${\tt biGr}(A)=\{\}=\mathcal Ess(A)$.  So assume not.

By Lemma~\ref{lemma:throwawaystuff}, whenever $(i,j)\not \in \mathcal Ess(A)$, there is some $(i',j')$ so that \[[i,j,r_A(i,j)]_b<[i',j',r_A(i',j')]_b.\]  
We may iteratively apply the Lemma~\ref{lemma:throwawaystuff} to construct a chain of inequalities
\[[i,j,r_A(i,j)]_b<[i',j',r_A(i',j')]_b< \ldots <[i'',j'',r_A(i'',j'')]_b\]  
with $(i'',j'')\in\mathcal Ess(A)$.
Therefore
\begin{align*}
A&=\vee\{[i,j,r_A(i,j)]_b:1\leq i,j\leq n]\} & \text{(by Lemma~\ref{lemma:bigrallentries})}\\
&=\vee \{[i,j,r_A(i,j)]_b: (i,j)\in \mathcal Ess(A)]\}.
\end{align*}
In particular, by (\ref{eqn:biGrcontain}) any biGrassmannian below $A$ has an upper bound in \[\{[i,j,r_A(i,j)]_b:(i,j)\in \mathcal Ess(A)\}.\]
\begin{claim}
\label{claim:antichain}
$\{[i,j,r_A(i,j)]_b:(i,j)\in \mathcal Ess(A)\}$ is an antichain, i.e. its elements are all incomparable.
\end{claim}
\begin{proof}
Take $(i,j),(i',j')\in \mathcal Ess(A)$.  Write $u=[i,j,r_A(i,j)]_b$ and $u'=[i',j',r_A(i',j')]_b$.

\noindent Case 1: $r_{u'}(i,j)\leq r$.

Since $u'\leq A$, we have $r_{u'}\geq r_A$.  In particular, $r_{u'}(i,j)\geq r$.  So $r_{u'}(i,j)=r$.  
By condition \ref{item:R2}, $r_{u'}(i-1,j),r_{u'}(i,j-1)\in \{r-1,r\}$ and $r_{u'}(i+1,j),r_{u'}(i,j+1)\in \{r,r+1\}$.  But since $(i,j)\in \mathcal Ess(A)$ and $r_{u'}\geq r_A$, applying Lemma~\ref{lemma:essentialrankchar} we are forced to have \[r_{u'}(i-1,j)=r_{u'}(i,j-1)=r=r_A(i-1,j)=r_A(i,j-1)\] and  \[r_{u'}(i+1,j)=r_{u'}(i,j+1)=r+1=r_{A}(i+1,j)=r_{A}(i,j+1).\]

Then $(i,j)\in \mathcal Ess(u')\{(i',j')\}$.  As such, $u'=u$.

\noindent Case 2: $r_{u'}(i,j)>r$.

Then $r_{u'}(i,j)>r_u(i,j)$.  So immediately, we conclude $u \not \geq u'$.

We may reverse the roles of $u$ and $u'$ in the above argument.  Therefore either $u$ and $u'$ are incomparable or $u=u'$.
\end{proof}

As a consequence of Claim~\ref{claim:antichain}, we have shown that $\{[i,j,r_A(i,j)]_b:(i,j)\in \mathcal Ess(A)\}$ is an antichain of biGrassmannian permutations whose least upper bound is $A$.  
Therefore, ${\tt biGr}(A)=\{[i,j,r_A(i,j)]_b:(i,j)\in \mathcal Ess(A)\}$.
\end{proof}

\begin{lemma}
\label{lemma:grassproject}
Suppose $A\in {\sf ASM}(n)$ and $u\in \mathcal G_n$.   If $r_A({\tt des}(u),j)\leq r_u({\tt des}(u),j)$ for all $j=1,\ldots, n$, then $u\leq A$.
\end{lemma}
\begin{proof}
 Let $i={\tt des}(u)$.
 Since $r_A(i,j)\leq r_u(i,j)$, we have $A\geq [i,j,r_u(i,j)]_b$.   Since $u$ is Grassmannian, all of its essential boxes occur in row $i$. Then by Lemma~\ref{lemma:bigrassranktest} we have $u'\leq A$ for all $u'\in \mathcal Ess(u)$.  So $A$ is an upper bound to $\mathcal Ess(u)$.  Then $A\geq u=\vee \mathcal Ess(u)$.
\end{proof}

With the above lemmas, we are now ready to prove Proposition~\ref{prop:therightasm}.
\begin{proof}[Proof of Proposition~\ref{prop:therightasm}]

\noindent (1)
Let $\mathcal Ess(A)=\{(i_1,j_1),\ldots, (i_k,j_k)\}$ and \[\beta^{(\ell)}=(i_\ell-r_A(i_\ell,j_\ell))\times(j_\ell-r_A(i_\ell,j_\ell))\]  as in (\ref{eqn:betapartition}).
 By construction, 
\begin{equation}
\label{eqn:bigrassencode}
[\beta^{(\ell)},i_\ell]_b=[i_\ell,j_\ell,r_A(i_\ell,j_\ell)]_b.
\end{equation}
Therefore, 
\begin{align*}
A&=\vee {\tt biGr}(A) & \text{(by (\ref{eqn:bigrjoin}))}\\
&=\vee\{[i,j,r_A(i,j)]_b:(i,j)\in \mathcal Ess(A)\} & \text{(by Proposition~\ref{prop:bigrassList})}\\
&=\vee\{[\beta^{(1)},i_1]_g,\ldots, [\beta^{(k)},i_k]_g\} & \text{(by (\ref{eqn:bigrassencode}))}\\
&=A_{\boldsymbol \beta_A,\mathbf b_A} & \text{(by (\ref{eqn:associatedasm})).}
\end{align*}

\noindent (2)
Let $u=[\lambda^{(A,i)},i]_g$.  
Since $\lambda^{(A,i)}=\lambda^{(u,i)}$, we must have 
\[m_A(i,j)=m_u(i,j) \quad \text{ for all } \quad j=1,\ldots,i.\]
Applying (\ref{eqn:trianglerank}), we have
\begin{equation}
\label{eqn:rankrowessentialsame}
r_A(i,j)=r_u(i,j) \text{ for all } j=1,\ldots, n.
\end{equation}

Let $\mathbf p_A=(i_1,\ldots,i_k)$ be the essential rows of $A$ and let $\boldsymbol \rho_A=(\lambda^{(A,i_1)}, \ldots, \lambda^{(A,i_k)})$  Then $\mathbf u_{\boldsymbol \rho_A,\mathbf p_A}=([\lambda^{(A,i_1)},i_1]_g, \ldots, [\lambda^{(A,i_k)},i_k]_g)$.

By (\ref{eqn:rankrowessentialsame}) and Lemma~\ref{lemma:grassproject}, $[\lambda^{(A,i_\ell)},i_\ell]_g\leq A$ for all $\ell=1,\ldots,k$.  As such, $A$ is an upper bound to $\mathbf u_{\boldsymbol \rho_A,\mathbf p_A}$ and hence 
\begin{equation}
\label{eqn:compare1}
A_{\boldsymbol \rho_A,\mathbf p_A}=\vee \mathbf u_{\boldsymbol \rho_A,\mathbf p_A}\leq A.
\end{equation}

On the other hand, by (\ref{eqn:rankrowessentialsame}), $r_{ [\lambda^{(A,i)},i]_g}(i,j)=r_A(i,j)$.  Then by Lemma~\ref{lemma:bigrassranktest}, 
\[[i,j,r_A(i,j)]_b\leq [\lambda^{(A,i)},i]_g \quad  \text{ for all } \quad 1\leq i,j\leq n.\]
  In particular, if $u\in {\tt biGr}(A)$, then there is some $i_\ell$ in the list $\mathbf p_A$ so that 
\[u\leq [\lambda^{(A_{i_\ell})},i_\ell]_g\leq A_{\boldsymbol \rho_A,\mathbf p_A}. \]  So $A_{\boldsymbol \rho_A,\mathbf p_A}$ is an upper bound to ${\tt biGr}(A)$ and hence 
\begin{equation}
\label{eqn:compare2}
A=\vee {\tt biGr}(A)\leq A_{\boldsymbol \rho_A,\mathbf p_A}.
\end{equation}

Therefore, by (\ref{eqn:compare1}) and (\ref{eqn:compare2}), $A=A_{\boldsymbol \rho_A,\mathbf p_A}$.
\end{proof}

We note that the parabolic model could also have been defined using a partition shape for \emph{every} row of $A$.  This has the drawback of having more redundant labels in each tableau.  However, the prism shapes have a direct connection to the poset of ASMs
\[A\leq B \quad \text{ if and only if } \quad \lambda^{(A,i)}\subseteq \lambda^{(B,i)} \quad \text{ for all } \quad i=1,\ldots,n.\]  This generalizes  the following description of the poset of Grassmannian permutations with a fixed descent.  Take $u,v\in \mathcal G_n$ with  ${\tt des}(u)={\tt des}(v)$.  Then
\[u\leq v \quad \text{ if and only if } \quad \lambda^{(u)}\subseteq \lambda^{(v)}.\]

\subsection{Inclusions of ASMs}

There is a natural inclusion $\iota:{\sf ASM}(n)\rightarrow {\sf ASM}(n+1)$ defined by \[A\mapsto 
\left(\begin{array}{c|c}
A& \mathbf 0\\
\hline
\mathbf 0&1
\end{array}
\right).
\]
We write \[{\sf ASM}(\infty):=\bigcup_{n=1}^\infty {\sf ASM}(n)/\sim\] where $\sim$ is the equivalence relation generated by $A\sim \iota(A)$.  Let \[\mathcal S_\infty=\bigcup_{n=1}^\infty \mathcal S_n/\sim.\] 
When context is clear,  we will freely identify an equivalence class its representatives.  We write $A\in {\sf ASM}(n)$ to indicate that $A$ has a representative which is an element of  ${\sf ASM}(n)$. 

Observe that
\begin{equation}
\label{eqn:asminf}
A\leq B \text{ if and only if } \iota(A)\leq \iota(B).  
\end{equation}
To see this, notice that $r_{\iota(A)}(i,n+1)=r_{\iota(A)}(n+1,i)=i$ for any $A\in {\sf ASM}(n)$.
Thus ${\sf ASM}(\infty)$ inherits the structure of a poset from the finite case.  In particular, for any $n$, there is an order embedding 
\[{\sf ASM}(n)\hookrightarrow {\sf ASM}(\infty).\]
To compare two classes in ${\sf ASM}(\infty)$, we may take $N$ large enough so that there are representatives in $A,B\in{\sf ASM}(N)$.  Due to (\ref{eqn:asminf}) the resulting order does not depend on the choice of $N$.  
Pairwise, joins and meets still exist so ${\sf ASM}(\infty)$ is a lattice.  However it is \emph{not} complete; in particular, the entire lattice ${\sf ASM}(\infty)$ has no upper bound.

Note that if $u\in \mathcal G_n$, we have \[(\lambda^{(u)},{\tt des}(u))=(\lambda^{(\iota(u))},{\tt des}(\iota(u))).\]  So the bijection in Lemma~\ref{lemma:grassbiject} is stable under inclusion.  Write $\mathcal G_\infty$ and $\mathcal B_\infty$ for the sets of Grassmannian and biGrassmannian permutations in $\mathcal S_\infty$.
Diagrams are also stable under inclusion, i.e. $D(A)=D(\iota(A))$.  
  Therefore
\begin{equation}
{\tt biGr}(\iota(A))=\{\iota(u):u\in{\tt biGr}(A)\}.
\end{equation}
Therefore, elements of ${\sf ASM}(\infty)$ are encoded by (finite) antichains in $\mathcal B_\infty$.

\subsection{Partial ASMs}

We now discuss another poset, which is closely related to ${\sf ASM}(n)$.
  A {\bf partial alternating sign matrix} is a matrix with entries in $\{-1,0,1\}$ so that
\begin{enumerate}
\item the nonzero entries in each row and column alternate in sign,
\item each row and column sums to 0 or 1, and
\item the first nonzero entry of any row or column is 1.
\end{enumerate}
 A {\bf partial permutation} is a partial ASM with entries in $\{0,1\}$.  Write ${\sf PA}(n)$ for the set of $n\times n$ partial ASMs and ${\sf P}(n)$ for the set of $n\times n$ partial permutation matrices. 
We sometimes say $A$ (or $w$) is an {\bf honest} ASM (or {\bf honest} permutation) to emphasize that $A\in {\sf ASM}(n)$ (or $w\in \mathcal S_n$).

As in the case of ASMs, we may endow ${\sf PA}(n)$ with the structure of a poset by comparison of corner sum functions.  
M.~Fortin studied ${\sf PA}(n)$, showing that it is the  Dedekind-MacNeille completion of ${\sf P}(n)$ \cite[Section~6]{fortin2008macneille}.  Here, partial permutation matrices are identified with \emph{partial injective functions}. The poset structure  defined by corner sum matrices agrees with the extended Bruhat order defined by L.~E.~Renner in \cite{renner2005linear}.

\begin{lemma}
\label{lemma:cancompletion}
 Every $A\in {\sf PA}(n)$ has a canonical completion to $\widetilde{A}\in {\sf ASM}(N)$, with $n\leq N\leq 2n$.  
\end{lemma}

\begin{proof}
The construction is similar to the one in for partial permutations found in \cite[Proposition~15.8]{miller2004combinatorial}.  Starting from the top row of $A$, if sum of row $i$ is zero, append a new column to $A$ with a 1 in the $i$th row.  Continue in this way from top to bottom.  Then starting from the leftmost column, if column $j$ sums to zero, add a new row with a 1 in position $j$.  Let $\widetilde{A}$ be the matrix obtained by this procedure.

By construction, $\widetilde{A}$ satisfies \ref{item:A1}; nonzero entries alternate in sign along rows and columns.  Also,  the entries of within each row and column of $\widetilde{A}$ sum to 1, so  \ref{item:A2} holds.  As such, the sum of all entries in $A$ counts the total number of rows, as well as the number of columns.  So $\widetilde{A}$ is square.   At most $n$ columns and rows were added.  So $\tilde A\in {\sf ASM}(N)$ for some $ n\leq N\leq 2n$.  
\end{proof}

\begin{example}
If $A=
\left(\begin{array}{cccc}
 0 & 0 &  0\\
 0&1&0\\
 1&-1&0
\end{array}\right)$ then
$\widetilde{A}=
\left(\begin{array}{ccc|cc}
 0 & 0 & 0 & 1& 0\\
 0 & 1 & 0&0&0\\
 1&-1&0&0&1\\\hline 
0 & 1 & 0 & 0&0\\
0 & 0 & 1 & 0&0
\end{array}\right)$.
Since the sum of the entries of $A$ is $1$, $N=2n-1=5$. \qed
\end{example}
For $w\in {\sf P}(n)$ we define the {\bf length} of $w$ to be $\ell(w):=\ell(\widetilde{w})$.  Similarly, we define the diagram $D(A):=D(\widetilde{A})$.  By construction, $D(A)$ is contained in the $n\times n$ grid.
\begin{lemma}
\label{lemma:rankpartialorder}
 $r_{A}\geq r_{B}$ if and only if $r_{\widetilde{A}}\geq r_{\widetilde{B}}$.
\end{lemma}
\begin{proof}
If $r_{\widetilde{A}}\geq r_{\widetilde{B}}$ it is immediate that $r_{A}\geq r_{B}$.

Now assume that  $r_{A}\geq r_{B}$.  By construction, the essential set of both $\widetilde{A}$ and $\widetilde{B}$ is contained in the first $n$ rows and columns.  So for any $(i,j)\in \mathcal Ess(\widetilde{A})$, we have $r_{\widetilde{B}}(i,j)\leq r_{\widetilde{A}}(i,j)$ and therefore $[i,j,r_{\widetilde{A}}(i,j)]_b\leq B$.  Then $\widetilde{B}$ is an upper bound to ${\tt biGr}(A)$ and so $A\leq B$ which implies $r_{\widetilde{A}}\geq r_{\widetilde{B}}$.
\end{proof}

Taking the inclusion of $\widetilde{A}$ into ${\sf ASM}(2n)$ is an order embedding ${\sf PA}(n)\hookrightarrow {\sf ASM}(2n)$.
As such, we may study the order on  ${\sf PA}(n)$ by identifying each partial ASM with its image under the above inclusion.

A {\bf partial biGrassmannian} is an element $b\in{\sf P}(n)$ so that $\widetilde{b}\in \mathcal S_{2n}$ is biGrassmannian.  Again, these are indexed by triples $(i,j,r)$ but we omit condition \ref{item:B3}.  Write $[i,j,r_{ij}]_b$ for the partial biGrassmannian in ${\sf P}(n)$.  By \cite{fortin2008macneille}, these are the basic elements of ${\sf PA}(n)$.

Notice, that restrictions of honest ASMs to northwest submatrices produce partial ASMs.  Take $A\in {\sf ASM}(N)$.  Then if $n\leq N$, we have $A_{[n],[n]}\in {\sf PA}(n)$.  Notice $A\leq  B$ implies $A_{[n],[n]}\leq B_{[n],[n]}$.  However, the converse certainly does not hold.  However, in the case $A=\widetilde{A}_{[n],[n]}$, we do have $A \leq B$ whenever $A_{[n],[n]}= B_{[n],[n]}$.  This follows since $\mathcal Ess(A)\subseteq n\times n$ and so $u\leq B$ for all $u\in {\tt biGr}(A)$.

\section{Subword complexes and prism tableaux}

\subsection{Simplicial complexes}

 Recall that $\mathbb P(S)$ denotes the power set of $S$.  
A {\bf simplicial complex} $\Delta$ is a subset of $\mathbb P([N])$ so that whenever $f\in \Delta$ and $f'\subseteq f$, we have $f'\in \Delta$.  An element $f\in \Delta$ is called a {\bf face}.   The {\bf dimension} of $f$ is $\dim(f)=|f|-1$.   Write \[\dim(\Delta)=\max\{\dim(f):f\in \Delta\}.\] If $f\in \Delta$, the {\bf codimension} of $f$ is ${\rm codim}(f)=\dim(\Delta)-\dim(f)$.
The set of faces of $\Delta$ ordered by inclusion form a poset.  Let 
\begin{equation}
\label{eqn:facet}
F(\Delta)={\tt MAX}(\Delta)
\end{equation}
 denote the set of {\bf facets} of $\Delta$, i.e. the maximal faces.
 Then define 
\begin{equation}
F_{\tt max}(\Delta)=\{f\in\Delta: {\rm codim}(f)=0\}.
 \end{equation}
Necessarily, $F_{\tt max}(\Delta)\subseteq F(\Delta)$.  When this containment is an equality, $\Delta$ is called {\bf pure}.

 Given two simplicial complexes $\Delta_1,\Delta_2 \subseteq \mathbb P([N])$, we may refer without ambiguity to the intersection (or union) of $\Delta_1$ and $\Delta_2$;  it is precisely their intersection (or union) as sets.  A straightforward verification shows that $\Delta_1\cap \Delta_2$ and $\Delta_1\cup \Delta_2$ are themselves simplicial complexes.  

\begin{lemma}
\label{lemma:facetoverlays}
Fix simplicial complexes  $\Delta_1,\ldots,\Delta_k\subseteq \mathbb P([N])$.  Let $\Delta=\Delta_1\cap\ldots\cap \Delta_k$.    Then \[F(\Delta)\subseteq \{f_1\cap\ldots \cap f_k:f_i\in F(\Delta_i)\}.\]
\end{lemma}
\begin{proof}
Fix $f\in F(\Delta)\subseteq \Delta$.  Then $f\in \Delta_i$ for all $i$.    
For each $i$, there exists some $f_i\in F(\Delta_i)$ such that $f\subseteq f_i$.  Therefore, 
\begin{equation}
f\subseteq  f_1\cap \ldots \cap f_k\subseteq f_i \text{ for all } i=1,\ldots, k.
\end{equation}
Then $ f_1\cap \ldots \cap f_k\in \Delta_i$ for all $i$.  As such, 
\begin{equation}
\label{eqn:facetcontainedsubset}
f\subseteq f_1\cap \ldots \cap f_k\in \Delta_1\cap  \ldots \cap \Delta_k=\Delta.
\end{equation}
Since $f\in F(\Delta)$, the containment in (\ref{eqn:facetcontainedsubset}) is actually an equality.
\end{proof}

Let $\Bbbk[\mathbf z]=\Bbbk[z_1,\ldots,z_N]$.  Given $\mathbf v=(v_1,\ldots,v_N)\in \mathbb N^{N}$, write $\mathbf z^{\mathbf v}:=\prod_{i=1}^Nz_i^{v_i}$.  If $\mathbf v\in \{0,1\}^N$, then $\mathbf z^{\mathbf v}$ is a {\bf square-free monomial}.   An ideal  is called a {\bf square-free monomial ideal} if it has a generating set of square-free monomials.
Stanley-Reisner theory describes the correspondence between square-free monomial ideals in $\Bbbk[\mathbf z]$ and simplicial complexes $\Delta\subseteq \mathbb P([N])$.  
We give a brief overview.  For more background, see \cite[Chapter 1]{miller2004combinatorial}.  

Notice square-free monomials in $\Bbbk[\mathbb z]$ correspond to faces in $\mathbb P([N])$.  Given $f\in \mathbb P([N])$, write $\displaystyle \mathbf z^f=\prod_{i\in f} z_i$.  
\begin{definition} 
\label{def:SRComplex}
The {\bf Stanley-Reisner ideal} of $\Delta$ is
   \[I_\Delta=\langle \mathbf z^f:f\not \in \Delta\rangle.\] The quotient $\Bbbk[\mathbf z]/I_{\Delta}$ is called the {\bf Stanley-Reisner ring} of $\Delta$.
\end{definition}
Write $\mathfrak m_f=\langle z_i:i\in f \rangle$ and let $\overline{f}=[N]-f$.
\begin{theorem}{{\cite[Theorem 1.7]{miller2004combinatorial}}}
\label{thm:bijSR}
The map $\Delta\mapsto  I_\Delta$ is a bijection between square-free monomial ideals in $\Bbbk[\mathbf z]$ and simplicial complexes $\Delta\subseteq \mathbb P([N])$.  $I_\Delta$ can be expressed as an intersection of monomial prime ideals
\begin{equation}
\label{eqn:primedecomp}
I_\Delta=\bigcap_{f\in \Delta}\mathfrak m_{\overline{f}}.
\end{equation}
\end{theorem}
Explicitly, the inverse  map takes a square-free monomial ideal $I$ to 
\[\Delta(I):=\{f\subseteq [N]: \mathbf z^f \not \in I\}.\]  Given a square-free monomial ideal $I$, we say $\Delta(I)$ is the {\bf Stanley-Reisner complex} associated to $I$.

The following lemma is straightforward from Definition~\ref{def:SRComplex}, but we give the details.
\begin{lemma}
\label{lemma:SRComplex}
Let $\{I_\alpha\}_{\alpha\in \mathcal A}$ be a set square-free monomial ideals with $I_\alpha\subseteq \Bbbk[z_1,\ldots, z_N]$. Then 
$\Delta(\sum_{\alpha\in \mathcal A}I_\alpha)=\bigcap_{\alpha\in \mathcal A} \Delta(I_\alpha)$.
\end{lemma}
\begin{proof}
A generating set for $\sum_{\alpha\in \mathcal A}I_\alpha$ can be obtained by concatenation of the generating sets for the $I_\alpha$'s.  So it is a square-free monomial ideal.  Notice a monomial $m\in \sum_{\alpha\in\mathcal A}I_\alpha$ if and only if $m\in I_\alpha$ for some $\alpha \in \mathcal A$.

Assume $f\subseteq[N]$.  Then,
\begin{align*}
f\in\Delta(\sum_{\alpha\in \mathcal A}I_\alpha)&\iff {\mathbf z}^f\not \in \sum_{\alpha\in \mathcal A}I_\alpha\\
&\iff {\mathbf z}^f \not \in I_\alpha \text{ for all } \alpha\in \mathcal A\\
&\iff f\in  \Delta(I_\alpha) \text{ for all } \alpha\in \mathcal A\\
&\iff f\in \bigcap_{\alpha\in \mathcal A} \Delta(I_\alpha). \qedhere
\end{align*}
\end{proof}

\subsection{Subword complexes}

We now recall the definition of a subword complex, following \cite{knutson2004subword}.  Let $\Pi$ be a Coxeter group minimally generated by simple reflections $\Sigma$.  A {\bf word} is an ordered list $Q=(s_{1},\ldots,s_{m})$ of simple reflections in $\Sigma$.   A {\bf subword} of $Q$ is an ordered subsequence $P=(s_{i_i},\ldots, s_{i_k})$.  Subwords of $Q$ are naturally identified with faces of the simplicial complex $\mathbb P([m])$.

A word  $P$ {\bf represents} $w\in \Pi$ if $w=s_{1}\cdots s_{m}$ and $\ell(w)=m$, i.e. the ordered product is a reduced expression for $w$. We say $Q$ {\bf contains} $w$ if $Q$ has a subword which represents $w$.  Then define the {\bf subword complex} \[\Delta(Q,w)=\{Q-P:P \text{ contains } w\}.\]  
We will abbreviate $\mathcal F_P:=Q-P$.   Immediately by definition, 
\begin{equation}
\mathcal F_P\subseteq \mathcal F_{P'} \quad \text{ if and only if } \quad P\supseteq P'.
\end{equation}
A well known characterization of the Bruhat order on $\mathcal S_n$ is via subwords: 
\begin{equation}
\label{eqn:bruhatword}
\text{$w\geq v$  if and only if some (and hence every) reduced word for $w$ contains $v$}.
\end{equation}  See \cite[Section 5.10]{humphreys1992reflection}.  This is equivalent to the order on $\mathcal S_n$ as defined in (\ref{eqn:asmposetdef}). See \cite[Theorem 2.1.5]{bjorner2006combinatorics} for a proof.

The {\bf Demazure algebra} of $(\Pi,\Sigma)$ over a ring $R$ is freely generated by $\{e_w:w\in \Pi\}$ with multiplication given by 
\[e_we_s=
\begin{cases}
e_{ws} & \text { if } \ell(ws)>\ell(w)\\
e_w &\text{ if } \ell(ws)<\ell(w).
\end{cases}
\]
If $Q=s_1 \ldots s_k$, the {\bf Demazure product} $\delta(Q)$ is defined by the product $e_{s_1}\cdots e_{s_m}=e_{\delta(Q)}$.  The faces of $\Delta(Q,w)$ have a natural description in terms of the Demazure product.
\begin{lemma}[{\cite[Lemma 3.4]{knutson2004subword}}]
\label{lemma:dez}
$\delta(P)\geq w$ if an only if $P$ contains $w$.  
\end{lemma}
Then
$
\Delta(Q,w)=\{\mathcal F_P:\delta(P)\geq w\}.$
This motivates the following definition.  Let $\Pi=\mathcal S_n$ and $\Sigma=\{(i,i+1):i=1,\ldots,n-1\}$ be the set of simple transpositions. Given $A\in {\sf ASM}(n)$, define 
\begin{equation}
\Delta(Q,A)=\{\mathcal F_P:\delta(P)\geq A\}.
\end{equation}
This is itself a simplicial complex, but need not be a subword complex.
  Immediately from the definition, 
\begin{equation}
\label{eqn:subwordcont}
\text{if $A\geq B$ then $\Delta(Q,A)\subseteq \Delta(Q,B)$.}
\end{equation}
We will show that $\Delta(Q,A)$ is a union of subword complexes.  In particular, if $A\in {\sf ASM}(m)$ with $m\leq n$ each of these subword complexes correspond to permutations in $\mathcal S_m$.

\begin{lemma}
\label{lemma:boundinrightplace}
Suppose $w\in \mathcal S_\infty$ is an upper bound to $\{w_1,\ldots,w_k\}\subseteq \mathcal S_m$.  Then there exists 
$w'\in S_m$ so that $ \vee \{w_1,\ldots,w_k\}\leq w'\leq w$.
\end{lemma}
\begin{proof}
Let $P$ be a reduced word for $w$.  By (\ref{eqn:bruhatword}), $P$ contains a subword, $P_i$ which represents $w_i$ for each $i$.  Let $P'=\bigcup_{i=1}^k P_i\subseteq P$.  By Lemma~\ref{lemma:dez}, since $P'$ contains each of the $w_i$'s, we have $\delta(P')\geq w_i$ for all $i$.  So $\delta(P')$ is an upper bound to  $\{w_1,\ldots,w_k\}$. 
Again, by Lemma~\ref{lemma:dez}, $P'$ contains $\delta(P')$ and hence $P$ contains $\delta(P')$.  So \[w=\delta(P)\geq \delta(P').\]
Finally, the word $\mathcal P'$ uses only simple transpositions from $S_m$, so $\delta(P')\in \mathcal S_m$.
\end{proof}
As a corollary, we obtain the following.
\begin{corollary}
\label{cor:perminfty}
\begin{enumerate}
\item ${\tt Perm}(A)={\tt MIN}(\{w\in \mathcal S_\infty:w\geq A\}).$
\item ${\tt Perm}(A)={\tt MIN}(\{w\in  {\sf P}(n):w\geq A\}).$
\end{enumerate}
\end{corollary}
\begin{proof}
\noindent (1) This is immediate from Lemma~\ref{lemma:boundinrightplace}.

\noindent (2) Fix $w\in {\sf P}(n)$.  Consider the inclusions $\widetilde{A},\widetilde{w}\in{\sf ASM}(2n)$.  Then $\widetilde{w}\geq A$ is an upper bound to ${\tt biGr}(\widetilde{A})={\tt biGr}(A)$.  Applying Lemma~\ref{lemma:boundinrightplace}, we obtain $w'\in \mathcal S_{n}$ with $\widetilde{A}\leq w'\leq \widetilde{w}$.  Since $w'\in \mathcal S_n$,we may take its representative $\widetilde{w'}\in {\sf ASM}(2n)$.  So $\widetilde{A}\leq \widetilde{w'}\leq \widetilde{w}$.  
Applying Lemma~\ref{lemma:rankpartialorder}, we see that
$A\leq w'\leq w$.  So the statement follows.
\end{proof}

\begin{proposition}
\label{prop:interesctunion}
Fix a word $Q$ and $A\in {\sf ASM}(n)$.  
\begin{enumerate}
\item $\displaystyle\Delta(Q,A)=\bigcup_{w\in {\tt Perm}(A)} \Delta(Q,w)$.
\item If $A=\vee\{A_1,\ldots,A_k\}$ then \[\Delta(Q,A)=\bigcap_{i=1}^k \Delta(Q,A_i).\]
\item $F(\Delta(Q,A))=\{\mathcal F_P: P \text{ is a reduced expression for some } w\in{\tt Perm}(A)\}.$
\end{enumerate} 
\end{proposition}

\begin{proof}
\noindent (1) Since $w\geq A$, applying (\ref{eqn:subwordcont}) we have
\[\Delta(Q,A)\supseteq\bigcup_{w\in {\tt Perm}(A)} \Delta(Q,w).\]
If $\mathcal F_P\in \Delta(Q,A)$ then $\delta(P)\geq A$.  By (\ref{eqn:PermA}) and Corollary~\ref{cor:perminfty},  there exists  $w\in {\tt Perm}(A)$ so that $\delta(P)\geq w\geq A$.  Then $\mathcal F_P\in \Delta(Q,w)$.  Therefore, $\displaystyle\Delta(Q,A)\subseteq \bigcup_{w\in {\tt Perm}(A)} \Delta(Q,w)$.

\noindent (2) Since $A\geq A_i$, applying  (\ref{eqn:subwordcont}), we have that $\Delta(Q,A)\subseteq \Delta(Q,A_i)$ for each $i=1,\ldots, k$.  So \[\Delta(Q,A)\subseteq \bigcap_{i=1}^k \Delta(Q,A_i).\]
If $\mathcal F_P\in \Delta(Q,A_i)$ for all $i$, then $\delta(P)\geq A_i$ for all $i$.  Since $A=\vee\{A_1,\ldots,A_k\}$ we must have $\delta(P)\geq A$.  So $\mathcal F_P\in \Delta(Q,A)$.

\noindent (3) Suppose $P$ is a reduced expression for some $w\in{\tt Perm}(A)$.  If $P$ contains $P'$ and $\mathcal F_{P'}\in \Delta(Q,A)$ then $w=\delta(P)\geq \delta(P')\geq A$.  By (\ref{eqn:PermA}), $\delta(P')=w$.  Since $P$ is a reduced expression for $w$, we have $P=P'$.  Therefore $\mathcal F_P\in \Delta(Q,w)$.
\end{proof}

For the rest of this section, we focus on a fixed ambient word $Q$.  Write $s_i$ for the simple transposition $(i,i+1)\in \mathcal S_{2n}$.  Define the {\bf square word} \[Q_{n\times n}=s_n \, s_{n-1} \, \ldots \, s_1 \quad s_{n+1} \,  s_{n} \, \ldots \, s_2 \quad \ldots \quad s_{2n-1} \, s_{2n-2} \, \ldots s_n.\]  
Order the boxes of the $n\times n$ grid by reading along rows from right to left, starting with the top row and working down to the bottom.  This ordering identifies each letter of $Q_{n\times n}$ with a cell in the $n\times n$ grid.

A {\bf plus diagram} is a subset of the $n\times n$ grid.  We indicate $(i,j)$ is in the plus diagram by marking its position in the grid with a $+$.  The identification of the letters in  $Q_{n\times n}$ with the grid defines a natural bijection between subwords of $Q_{n\times n}$ and plus diagrams.  As such,
we freely identify each word with its plus diagram. 

\begin{example}
When $n=3$, we have 
\[Q_{n\times n}= s_3 s_2 s_1 s_4 s_3 s_2 s_5 s_4 s_3.\]
Below, we label the entries of the $3\times 3$ grid with their corresponding simple transpositions.  We also give a subword of $Q_{3\times 3}$ and  its corresponding plus diagram.  
\[\begin{array}{ccc}
s_1 &s_2 &s_3 \\
 s_2 &s_3 &s_4 \\
  s_3 &s_4 &s_5
\end{array}\hspace{3em} 
 s_3 \, -\, -\,  -\, s_3\, s_2 \,- \,-\, \,-  \hspace{3em} 
 \begin{array}{ccc}
\cdot &\cdot &+ \\
+&+&\cdot \\
  \cdot &\cdot &\cdot
\end{array}  \]
Notice that $P$ is not a reduced expression, $s_3s_3s_2=s_2$.  Therefore, it is not a facet of  $\Delta(Q_{n\times n},A)$ for any $A\in {\sf ASM}(n)$. \qed
\end{example}

For brevity, write  $\Delta_A:=\Delta(Q_{n\times n},A)$.  
Assign $\mathcal F_P$ the weight
\[{\tt wt}(\mathcal F_P)=\prod_{i=1}^n x_i^{n_i} \quad \text{ where} \quad n_i=|\{j:(i,j)\in P\}|.\]
When $w\in \mathcal S_n$, the complex $\Delta_w$ is a pure simplicial complex.  Its facets are in immediate bijection with pipe dreams (also known as RC-graphs). 
\begin{theorem}[\cite{Fomin.Kirillov,Bergeron.Billey,knutson2005grobner}]
\label{thm:schubequation}
\begin{equation}
\label{eqn:schubertsubwordexpression}
\mathfrak S_w=\sum_{\mathcal F_P\in  F(\Delta_w)}{\tt wt}(\mathcal F_P).\end{equation}
\end{theorem}
For permutations, $\Delta_w$ is the Stanley-Reisner complex of a degeneration of the Schubert determinantal ideal $I_w$ \cite[Theorem~B]{knutson2005grobner}.  The same holds for $I_A$ and $\Delta_A$, see Section~\ref{subsection:asmdet} for details.

As a consequence of Theorem~\ref{thm:schubequation}, we have the following corollary.
\begin{corollary}
\label{cor:schubsumposet}
\begin{enumerate}
\item
$\displaystyle \sum_{w\in{\tt Perm}(A)} \mathfrak S_w=\sum_{\mathcal F_P\in F(\Delta_A)} {\tt wt}(\mathcal F_P).$
\item $\displaystyle \sum_{w\in{\tt MinPerm}(A)} \mathfrak S_w=\sum_{\mathcal F_P\in F_{\tt max}(\Delta_A)} {\tt wt}(\mathcal F_P).$
\end{enumerate}
\end{corollary}
\begin{proof}
\noindent (1)  By Proposition~\ref{prop:interesctunion},  
\begin{equation}
\label{eqn:facetdisjoint}
F(\Delta_A)=\bigcup_{w\in {\tt Perm}(A)} F(\Delta_w).
\end{equation}
   $\mathcal F_P$ is a facet of $\Delta_w$ if and only if $P$ represents $w$.  A subword can represent at most one permutation, so the union in (\ref{eqn:facetdisjoint}) is disjoint.
Therefore, applying (\ref{eqn:schubertsubwordexpression}), we have \[ \sum_{w\in{\tt Perm}(A)} \mathfrak S_w=\sum_{w\in {\tt Perm}(A)} \sum_{\mathcal F_P\in F(\Delta_w)} {\tt wt}(\mathcal F_P)=\sum_{\mathcal F_P\in F(\Delta_A)} {\tt wt}(\mathcal F_P).\]

\noindent (2)  Observe that \[F_{\tt max}(\Delta_A)=\bigcup_{w\in {\tt MinPerm}(A)} F(\Delta_w).\]  Again this union is disjoint.  So the result follows.
\end{proof}

\subsection{Proof of Theorem~\ref{theorem:schubertsum}}
\label{subsection:biject}

Take $T\in {\tt RSSYT}(\lambda,d)$ and write $T_{ij}$ for the entry of $T$ which is in the $i$th row and $j$th column in the ambient $\mathbb Z_{> 0}\times \mathbb Z_{> 0}$ grid (as in (\ref{eqn:tableauposition})).  Define a plus diagram $ P_T$ by placing a plus in position $(T_{ij},i+j-T_{ij})$ for each label in $T$.  This in turn defines a map $T\mapsto \mathcal F_{P_T}$.  Define $\Phi_{\lambda,d}(T)=\mathcal F_{P_T}$. 

\begin{proposition}
\label{prop:tabbiject}
$\Phi_{\lambda,d}:{\tt RSSYT}(\lambda,d)\rightarrow  F(\Delta_{u_{\lambda,d}})$ is a bijection.
\end{proposition}
Proposition~\ref{prop:tabbiject} is well known.  For a proof, see e.g. \cite[Proposition~5.3]{KMY}.
Define 
\begin{equation}
\label{eqn:prismmap}
\Phi_{\boldsymbol \lambda,\mathbf d}:{\tt AllPrism}({\boldsymbol \lambda,\mathbf d})\rightarrow \Delta_{A_{\boldsymbol \lambda,\mathbf d}}
\end{equation}
where
\begin{equation}
\label{eqn:prismmapdef}
\Phi_{\boldsymbol \lambda,\mathbf d}(T^{(1)},\ldots, T^{(k)})=\Phi_{\lambda^{(1)},d_1}( T^{(1)})\cap \ldots \cap \Phi_{\lambda^{(k)},d_k}( T^{(k)}).
\end{equation}
 Equivalently, $\Phi_{\boldsymbol \lambda,\mathbf d}(T^{(1)},\ldots, T^{(k)})=\mathcal F_{P_{\mathcal T}}$ where $P_{\mathcal T}=\bigcup_{i=1}^k P_{T^{(i)}}$. By  part (2) of Proposition~\ref{prop:interesctunion}, $\Phi_{\boldsymbol \lambda,\mathbf d}$ is well defined. 
\begin{example}
Continuing Example~\ref{example:prism}, we have
\ytableausetup{boxsize=.95em, aligntableaux=center}
\[\mathcal T=\begin{tikzpicture}[x=1em,y=1em,baseline=4.2em]
\draw[step=1,gray!30,very thin] (0,1) grid (7,8); 
\node[] at (1.5,5){
\begin{ytableau}
\none\\
{\color{Rhodamine}1}\\
\none\\
{\color{Plum}1}{\color{blue}1}&{\color{Plum}1}\\
{\color{Plum}3}{\color{blue}2}&{\color{Plum}3}&{\color{Plum}2}\\
{\color{blue}6}&{\color{blue}3}\\
\end{ytableau}};
\end{tikzpicture} 
\hspace{2em}
 \mapsto \hspace{2em}
P_{\mathcal T}=\begin{array}{ccccccc}
\cdot &+ &\cdot &+ &+ &\cdot &\cdot \\
\cdot &\cdot &\cdot&+ &\cdot &+ &\cdot \\
\cdot &\cdot &+ &+ &+ &\cdot &\cdot \\
\cdot &\cdot &\cdot &\cdot &\cdot &\cdot &\cdot \\
\cdot &\cdot &\cdot &\cdot &\cdot &\cdot &\cdot \\
+ &\cdot &\cdot &\cdot &\cdot &\cdot &\cdot \\
\cdot &\cdot &\cdot &\cdot &\cdot &\cdot &\cdot 
\end{array} 
 \] \qed
\end{example}

\begin{lemma}
\label{lemma:weightpreserve}
$\Phi_{\boldsymbol \lambda,\mathbf d}$ is weight preserving.
\end{lemma}
\begin{proof}
The plus diagram $P_{T}$ has a plus in position $(i,j)$ if and only if there is some $a$ so that $T_{a,i+j-a}=i$.

Let $\mathcal F_P=\Phi_{\boldsymbol \lambda,\mathbf d}(\mathcal T)$.  Then $P=P_{T^{(1)}}\cup \ldots \cup P_{T^{(k)}}$.  Therefore,
$(i,j)\in P$ if and only if $(i,j)\in P_{T^{(\ell)}} \text{ for some } \ell\in [k]$.  As such, $(i,j)\in P$ if and only if the label $i$ appears in the $(i+j)$th antidiagonal of $\mathcal T$.  Therefore,
\begin{align*}
{\tt wt}(\mathcal F_P)=\prod_{(i,j)\in P} x_i =\prod_{i}x_i^{n_i} ={\tt wt}(\mathcal T). & \qedhere
\end{align*}
\end{proof}

 Notice by Lemma~\ref{lemma:facetoverlays} and Proposition~\ref{prop:tabbiject},
\begin{equation}
\label{eqn:facetsurject}
F(\Delta_{A_{\boldsymbol \lambda,\mathbf d}})\subseteq \Phi_{\boldsymbol \lambda,\mathbf d}({\tt AllPrism}(\boldsymbol \lambda,\mathbf d)).
\end{equation}
If $\mathcal T$ is minimal,  $\Phi_{\boldsymbol \lambda,\mathbf d}(\mathcal T)\in F_{\tt max}(\Delta_{A_{\boldsymbol \lambda,\mathbf d}}).$
This implies
\begin{equation}
\label{eqn:twodegrees}
{\tt deg}(\boldsymbol \lambda,\mathbf d)={\tt deg}(A_{\boldsymbol \lambda,\mathbf d}).
\end{equation}
For permutation matrices, ${\tt Perm}(w)=w$ and so ${\tt deg}(w)=\ell(w)$.  This shows the original definition for a minimal prism tableau given in \cite{weigandt2015prism} agrees with the definition stated here.

Call $\mathcal T$ {\bf facet} if $\Phi_{\boldsymbol \lambda,\mathbf d}(\mathcal T)\in  F(\Delta_A)$.  Let ${\tt Facet}(\boldsymbol \lambda,\mathbf d)\subseteq {\tt AllPrism}(\boldsymbol \lambda,\mathbf d)$ denote the set of facet prism tableaux.  
Write ${\tt StableFacet}(\boldsymbol \lambda,\mathbf d)\subseteq {\tt Facet}(\boldsymbol \lambda,\mathbf d)$ for the set of facet tableaux which have no unstable triples.
By (\ref{eqn:facetsurject}), 
\begin{equation}
{\tt Perm}({A_{\boldsymbol \lambda,\mathbf d}})=\{w:\Phi_{\boldsymbol \lambda,\mathbf d}(\mathcal T)\in \Delta_w \text{ for some } \mathcal T\in {\tt Facet}(\boldsymbol \lambda,\mathbf d)\}.
\end{equation}

\begin{example}
\label{example:facet}
Let $A$ be as in Example~\ref{example:noneqi}.  Set $\boldsymbol \lambda=((2),(2))$ and $\mathbf d=(1,2)$.  Notice that $\mathbb S(\boldsymbol \lambda,\mathbf d)$ is \emph{both} the parabolic and the biGrassmannian prism shape for $A$.  So $A_{\boldsymbol \lambda,\mathbf d}=A$.  
There are three prism fillings of $\mathbb S(\boldsymbol \lambda,\mathbf d)$, listed below.
\ytableausetup{boxsize=.95em, aligntableaux=top,nobaseline}
\[\mathcal T_1=
\begin{tikzpicture}[x=1em,y=1em,baseline=2.8em]
\draw[step=1,gray!30,very thin] (0,1) grid (4,5); 
\node[] at (1,4){
\begin{ytableau}
{\color{Plum} 1}&{\color{Plum} 1}\\
{\color{blue} 1}&{\color{blue} 1}
\end{ytableau}};
\end{tikzpicture}
\hspace{2em}
\mathcal T_2=
\begin{tikzpicture}[x=1em,y=1em,baseline=2.8em]
\draw[step=1,gray!30,very thin] (0,1) grid (4,5); 
\node[] at (1,4){
\begin{ytableau}
{\color{Plum} 1}&{\color{Plum} 1}\\
{\color{blue} 2}&{\color{blue} 1}
\end{ytableau}};
\end{tikzpicture}
\hspace{2em}
\mathcal T_3=
\begin{tikzpicture}[x=1em,y=1em,baseline=2.8em]
\draw[step=1,gray!30,very thin] (0,1) grid (4,5); 
\node[] at (1,4){
\begin{ytableau}
{\color{Plum} 1}&{\color{Plum} 1}\\
{\color{blue} 2}&{\color{blue} 2}
\end{ytableau}};
\end{tikzpicture}
\]
Only $\mathcal T_1$ is minimal, so ${\tt Prism}(\boldsymbol \lambda,\mathbf d)=\{\mathcal T_1\}$.  Therefore $\mathfrak A_{\boldsymbol \lambda,\mathbf d}={\tt wt}(\mathcal T_1)=x_1^3$.
The above prism tableaux correspond to the following plus diagrams.
\[ P_1=\begin{array}{cccc}
+ &+ &+ &\cdot\\
\cdot &\cdot &\cdot &\cdot\\
\cdot &\cdot &\cdot &\cdot\\
\cdot &\cdot &\cdot &\cdot
\end{array} 
\hspace{2em}
 P_2=\begin{array}{cccc}
+ &+ &+ &\cdot\\
+ &\cdot &\cdot &\cdot\\
\cdot &\cdot &\cdot &\cdot\\
\cdot &\cdot &\cdot &\cdot
\end{array} 
\hspace{2em}
 P_3=\begin{array}{cccc}
+ &+ &\cdot &\cdot\\
+ &+ &\cdot &\cdot\\
\cdot &\cdot &\cdot &\cdot\\
\cdot &\cdot &\cdot &\cdot
\end{array}
\]
Since $P_2\supsetneq P_1$, we have $\mathcal F_{P_2}\subsetneq \mathcal F_{P_1}$.  Therefore $\mathcal T_2$ is not a facet prism tableau. 
There are no plus diagrams in the image of $\Phi_{\boldsymbol \lambda,\mathbf d}$ which are strictly contained in $P_1$ or $P_3$, so by (\ref{eqn:facetsurject}), $\mathcal F_{P_1},\mathcal F_{P_3}\in F(\Delta_{A})$.  So $\mathcal T_1, \mathcal T_3\in {\tt Facet}(\boldsymbol \lambda, \mathbf d)$.

The word corresponding to $P_1$ is $s_3s_2s_1~=~4123$ and the word for $P_3$ is $s_2s_1s_3s_2=3412$.  Therefore ${\tt Perm}(A)=\{4123,3124\}$ and ${\tt MinPerm}(A)=\{4123\}$.  
\qed
\end{example}

\begin{theorem}
\label{thm:wtbijection}
\begin{enumerate}
\item $ F(\Delta_{A_{\boldsymbol \lambda,\mathbf d}})$  is in weight preserving bijection with ${\tt StableFacet}(\boldsymbol \lambda,\mathbf d)$.
\item The bijection in (1) restricts to a bijection between $ F_{\tt max}(\Delta_{A_{\boldsymbol \lambda,\mathbf d}})$ and ${\tt Prism}(\boldsymbol \lambda,\mathbf d)$.
\end{enumerate}
\end{theorem}
Theorem~\ref{theorem:schubertsum} follows as an immediate consequence of Theorem~\ref{thm:wtbijection} and Corollary~\ref{cor:schubsumposet}.
\begin{equation}
 \mathfrak A_{\boldsymbol \lambda,\mathbf d}=\sum_{T\in {\tt Prism}(\boldsymbol \lambda,\mathbf d)} {\tt wt}(T)=\sum_{w\in {\tt MinPerm}(A_{\boldsymbol \lambda,\mathbf d})}\mathfrak S_w. 
\end{equation}
Similarly, we have
\begin{equation}
\sum_{\mathcal T\in {\tt StableFacet}(\boldsymbol \lambda,\mathbf d)}{\tt wt}(\mathcal T)=\sum_{w\in{\tt Perm}(A_{\boldsymbol \lambda,\mathbf d})} \mathfrak S_w.
\end{equation}

For our proof of Theorem~\ref{thm:wtbijection}, we analyze the fibers of $\Phi_{\boldsymbol \lambda,\mathbf d}$
\begin{equation}
\Phi_{\boldsymbol \lambda,\mathbf d}^{-1}(\mathcal F_P)=\{\mathcal T\in {\tt AllPrism}(\boldsymbol \lambda,\mathbf d):\Phi_{\boldsymbol \lambda,\mathbf d}(\mathcal T)=\mathcal F_P\} .\end{equation}
For an arbitrary face of $\Delta_{A_{\boldsymbol \lambda,\mathbf d}}$, this fiber may be empty.  However, by (\ref{eqn:facetsurject}), facets have nonempty fibers. 
In Proposition~\ref{prop:otherlatticethings} we show that the fiber of any facet has the structure of a lattice.  Furthermore, the maximum element of $\Phi_{\boldsymbol \lambda,\mathbf d}^{-1}(\mathcal F_P)$ is the only tableau in the fiber with no unstable triples.  

Order ${\tt RSSYT}(\lambda,d)$ and ${\tt AllPrism}(\boldsymbol \lambda,\mathbf u)$ by entrywise comparison.
\begin{proposition}
\begin{enumerate}
\item ${\tt RSSYT}(\lambda,d)$ is a lattice.
\item ${\tt AllPrism}(\boldsymbol \lambda,\mathbf u)$ is a lattice.
\end{enumerate}
\end{proposition}
\begin{proof}
\noindent (1) Given $T,U\in{\tt RSSYT}(\lambda,d)$, we claim $T\wedge U=\boldsymbol\min(T,U)$ and $T\vee U=\boldsymbol\max(T,U)$.  Note that 
\begin{equation}
\text{if $a_1\leq a_2$ and $b_1\leq b_2$ then $\min\{a_1,b_1\}\leq \min\{a_2,b_2\}$}.
\end{equation}
  Similarly, 
\begin{equation}
\text{
if $a_1< a_2$ and $b_1< b_2$ then $\min\{a_1,b_1\}< \min\{a_2,b_2\}$.}
\end{equation}  The same statements hold when replacing $\min$ with $\max$.   Therefore, conditions \ref{item:T1} and \ref{item:T2} are preserved under taking entrywise minimums and maximums.  Furthermore, $\boldsymbol\min(T,U)$ and $\boldsymbol\max(T,U)$ use only labels from $[d]$.  Then 
\[\boldsymbol\min(T,U),\boldsymbol\max(T,U)\in {\tt RSSYT}(\lambda,d).\]  By applying Lemma~\ref{lemma:latticesituation}, we see that ${\tt RSSYT}(\lambda,d)$ is a lattice.

\noindent (2)  By (1), ${\tt AllPrism}(\boldsymbol \lambda,\mathbf d)$ is a product of lattices.   So  ${\tt AllPrism}(\boldsymbol \lambda,\mathbf d)$ is itself a lattice.  Again,  $\mathcal T\wedge \mathcal U=\boldsymbol\min(\mathcal T,\mathcal U)$ and $\mathcal T\vee \mathcal U=\boldsymbol\max(\mathcal T,\mathcal U)$.
\end{proof}
 Write 
\begin{equation}{\tt RSSYT}_P(\lambda,d):=\{T\in{\tt RSSYT}(\lambda,d):\Phi_{\lambda,d}(T)\supseteq \mathcal F_P\}.
\end{equation}

\begin{lemma}
\label{lemma:support}
\begin{enumerate}
\item Suppose $T,U\in{\tt RSSYT}_P(\lambda,d)$.  Then 
 \begin{equation*}
%\label{eqn:imagecontain1}
\Phi_{ \lambda, d}( T\vee  U)\supseteq \mathcal F_P \quad \text{ and } \quad \Phi_{ \lambda, d}( T\wedge  U)\supseteq \mathcal F_P .
\end{equation*}
As such, ${\tt RSSYT}_P(\lambda,d)$ is a lattice.
\item Suppose $T,U\in{\tt RSSYT}_P(\lambda,d)$  with $T<U$.  Then there exists $V\in {\tt RSSYT}_P(\lambda,d)$ so that $T<V\leq U$ and $V$ differs from $T$ by increasing the value of a single entry. 
\item Take $\mathcal T,\mathcal U\in \Phi_{\boldsymbol \lambda,\mathbf d}^{-1}(\mathcal F_P)$. Then
 \begin{equation*}
%\label{eqn:imagecontain}
\Phi_{\boldsymbol \lambda,\mathbf d}(\mathcal T\vee \mathcal U)\supseteq \mathcal F_P \quad \text{ and } \quad \Phi_{\boldsymbol \lambda,\mathbf d}(\mathcal T\wedge \mathcal U)\supseteq \mathcal F_P .
\end{equation*}

\end{enumerate}
\end{lemma}
\begin{proof}

\noindent (1)  Let $T,U\in{\tt RSSYT}_P(\lambda,d)$.  Fix an antidiagonal $D$ of $\lambda$. Let $\{a_1,a_2,\ldots, a_m\}$ and $\{b_1,b_2,\ldots,b_m\}$ be the ordered lists of labels which appear in antidiagonal $D$ of $T$ and $U$ respectively.  Then the entries in antidiagonal $D$ of $T\vee U$ are \[\{\max(a_1,b_1),\ldots, \max(a_m,b_m)\}\subseteq\{a_1,\ldots,a_m\}\cup \{b_1,\ldots,b_m\}.\]  Since this holds for every antidiagonal, 
\[P_{T\vee U}\subseteq P_T\cup P_U.\]  Therefore,
\[\Phi_{\lambda,d}(T\vee U)\supseteq \Phi_{\lambda,d}(T)\cap \Phi_{\lambda,d}(U)\supseteq \mathcal F_P.\]
The argument for $T\wedge  U$ is the same.

\noindent (2)  Suppose $T,U\in {\tt RSSYT}_P(\lambda,d)$ and $T<U$.  
  Since $T<U$ all entries of $T$ are (weakly) less than the entries of $U$.  Let $S=\{(i,j):T_{ij}<U_{ij}\}$.  Since $T\neq U$, there is some entry of $T$ that is strictly less than the corresponding entry in $U$, so $S\neq \emptyset$.

Since $S$ is finite and nonempty, there is some $(i,j)\in S$ so that $(i+1,j),(i,j-1)\not \in S$. 
Then replace the $(i,j)$ entry of $T$ with $U_{ij}$ and call this $V$.  Then $V\in {\tt RSSYT}(\lambda,d)$.  Furthermore, the entries in each antidiagonal of $V$ form a subset of the union of the antidiagonal entries of $T$ and $U$ so  $\Phi_{\lambda,d}(V)\supseteq \Phi_{\lambda,d}(T)\cap \Phi_{\lambda,d}(U)\supseteq \mathcal F_P$.
Then we have produced $V\in {\tt RSSYT}_P(\lambda,d)$ so that $T<V<U$.

\noindent (3) 
By definition, $\Phi_{\boldsymbol \lambda,\mathbf d}(\mathcal T)=\Phi_{\boldsymbol \lambda,\mathbf d}(\mathcal U)=\mathcal F_P$.  Then $\Phi_{\lambda^{(i)},d_i}(T^{(i)})\supseteq \mathcal F_P$    and $\Phi_{\lambda^{(i)},d_i}(U^{(i)})\supseteq \mathcal F_P$ for all $i=1,\ldots,k$.  Applying (1), we have \[\Phi_{\lambda^{(i)},d_i}(T^{(i)}\vee U^{(i)})\supseteq \mathcal F_P \text{ and } \Phi_{\lambda^{(i)},d_i}(T^{(i)}\wedge U^{(i)})\supseteq \mathcal F_P.\]
Therefore, \[\Phi_{\boldsymbol \lambda,\mathbf d}(\mathcal T\vee \mathcal U)=\Phi_{\lambda^{(1)},d_1}(T^{(1)}\vee U^{(1)})\cap \ldots \cap \Phi_{\lambda^{(k)},d_k}(T^{(k)}\vee U^{(k)})\supseteq \mathcal F_P\] and 
\begin{align*}
\Phi_{\boldsymbol \lambda,\mathbf d}(\mathcal T\wedge \mathcal U)=\Phi_{\lambda^{(1)},d_1}(T^{(1)}\wedge U^{(1)})\cap \ldots \cap \Phi_{\lambda^{(k)},d_k}(T^{(k)}\wedge U^{(k)})\supseteq \mathcal F_P. & \qedhere
\end{align*}
\end{proof}

\begin{proposition}
\label{prop:otherlatticethings}
 Fix $\mathcal F_P\in F(\Delta_{A_{\boldsymbol \lambda,\mathbf d}})$.
\begin{enumerate}
\item $\Phi_{\boldsymbol \lambda,\mathbf d}^{-1}(\mathcal F_P)$ is a lattice. 
\item Suppose $\mathcal T,\mathcal U\in \Phi_{\boldsymbol \lambda,\mathbf d}^{-1}(\mathcal F_P)$ with $\mathcal T<\mathcal U$. Then $\mathcal T$ has an unstable triple.
\item $|{\tt StableFacet}(\boldsymbol \lambda,\mathbf d)\cap \Phi_{\boldsymbol \lambda,\mathbf d}^{-1}(\mathcal F_P)|=1.$
\end{enumerate}
\end{proposition}
\begin{proof}
\noindent (1) $\Phi_{\boldsymbol \lambda,\mathbf d}^{-1}(\mathcal F_P)$ is a subposet of ${\tt AllPrism}(\boldsymbol \lambda,\mathbf d)$.  So it is enough to show that $\Phi_{\boldsymbol \lambda,\mathbf d}^{-1}(\mathcal F_P)$ is closed under taking joins and meets.

By Lemma~\ref{lemma:support}, $\Phi_{\boldsymbol \lambda,\mathbf d}(\mathcal T\wedge \mathcal U)\supseteq \mathcal F_P$.   Since $\mathcal F_P\in F(\Delta_{A_{\boldsymbol \lambda,\mathbf d}})$, this  containment  is actually an equality.  Therefore, $\Phi_{\boldsymbol \lambda,\mathbf d}(\mathcal T\wedge \mathcal U)\in \Phi_{\boldsymbol \lambda,\mathbf d}^{-1}(\mathcal F_P)$, i.e., $\Phi_{\boldsymbol \lambda,\mathbf d}^{-1}(\mathcal F_P)$ is closed under joins.  The argument for meets is the same.  So we conclude $\Phi_{\boldsymbol \lambda,\mathbf d}^{-1}(\mathcal F_P)$ is a lattice.

\noindent (2)
Suppose $\mathcal U>\mathcal T$. In particular, for some $i$, we have $U^{(i)}>T^{(i)}$.  By the part 2 of Lemma~\ref{lemma:support}, there is $V\in {\tt RSSYT}_P(\lambda^{(i)},d_i)$ with $U^{(i)}\geq V >T^{(i)}$ so that $V$ differs from $T^{(i)}$ by increasing the value a single entry.

Since $\Phi_{\lambda^{(\ell)},d_\ell}(T^{(\ell)})\supseteq \mathcal F_P$ for all $\ell=1,\ldots,k$ and $\Phi_{\lambda^{(i)},d_i}(V)\supseteq \mathcal F_P$, we have 
\begin{equation}
\label{eqn:here}
\Phi_{\boldsymbol \lambda,\mathbf d}(T^{(1)}, \ldots, T^{(i-1)}, V , T^{(i+1)},\ldots T^{(k)}) \supseteq \mathcal F_P.
\end{equation}
Since $\mathcal F_P$ is a facet, (\ref{eqn:here}) is an equality.  Then \[(T^{(1)}, \ldots, T^{(i-1)}, V, T^{(i+1)},\ldots T^{(k)})\in \Phi_{\boldsymbol \lambda,\mathbf d}^{-1}(\mathcal F_P).\]  So $\mathcal T$ has an unstable triple.

\noindent (3)  By (1), $\Phi_{\boldsymbol \lambda,\mathbf d}^{-1}(\mathcal F_P)$ is a lattice.  In particular, it is finite and nonempty so has a unique maximum element.

    By part (2), if $\mathcal T$ is not the maximum of $\Phi_{\boldsymbol \lambda,\mathbf d}^{-1}(\mathcal F_P)$, then it has an unstable triple.  Conversely, if $\mathcal T$ has an unstable triple, then by definition, there is $\mathcal T' \in \Phi_{\boldsymbol \lambda,\mathbf d}^{-1}(\mathcal F_P)$ with $\mathcal T<\mathcal T'$.  So $\mathcal T$ is not the maximum.
Therefore, ${\tt StableFacet}(\boldsymbol \lambda,\mathbf d)\cap \Phi_{\boldsymbol \lambda,\mathbf d}^{-1}(\mathcal F_P)$ is the maximum of $\Phi_{\boldsymbol \lambda,\mathbf d}^{-1}(\mathcal F_P)$.  So
 \[|{\tt StableFacet}(\boldsymbol \lambda,\mathbf d)\cap \Phi_{\boldsymbol \lambda,\mathbf d}^{-1}(\mathcal F_P)|=1. \qedhere\]
\end{proof}

\begin{proof}[Proof of Theorem~\ref{thm:wtbijection}]
\noindent (1)
Define $\Psi: F(\Delta_{A_{\boldsymbol \lambda,\mathbf d}})\rightarrow {\tt StableFacet}(\boldsymbol \lambda,\mathbf d)$ by mapping
$\mathcal F_P$ to the unique element in  ${\tt StableFacet}(\boldsymbol \lambda,\mathbf d)\cap \Phi_{\boldsymbol \lambda,\mathbf d}^{-1}(\mathcal F_P)$.  By Proposition~\ref{prop:otherlatticethings} part (3), this is well defined.  Injectivity follows since \[\Phi_{\boldsymbol \lambda,\mathbf d}^{-1}(\mathcal F_P)\cap\Phi_{\boldsymbol \lambda,\mathbf d}^{-1}(\mathcal F_{P'})=\emptyset\]whenever $P\neq P'$.

Given $\mathcal T\in {\tt StableFacet}(\boldsymbol \lambda,\mathbf d)$, let $\mathcal F_P=\Phi_{\boldsymbol \lambda,\mathbf u}$.  By the definition of ${\tt StableFacet}(\boldsymbol \lambda,\mathbf d)$, we have $\mathcal F_P\in  F(\Delta_{A_{\boldsymbol \lambda,\mathbf d}})$.  Then $\mathcal T=\Psi(\mathcal F_P)$. As such, $\Psi$ is surjective.

Since $\Phi_{\boldsymbol \lambda,\mathbf u}$ is weight preserving, $\Phi(F_{\tt max}(\Delta_{A_{\boldsymbol \lambda,\mathbf d}}))={\tt Prism}(\boldsymbol \lambda,\mathbf d)$.
\end{proof}

\section{ Multidegrees and ASM varieties}
\label{section:asmgeo}

\subsection{Multidegrees}

In this section, we review multidegrees.  See \cite[Chapter~8]{miller2004combinatorial} for an introduction.   
 We say $\Bbbk[\mathbf z]$ is {\bf multigraded} by $\mathbb Z^n$ if there is a semigroup homomorphism 
${\tt d}:\mathbb N^N\rightarrow \mathbb Z^n.$  We may interpret ${\tt d}$ as a map from monomials in $\Bbbk[\mathbf z]$ to elements of $ \mathbb Z^n$.  As such, we write ${\tt d}(\mathbf z^{\mathbf v}):={\tt d}(\mathbf v)$.

Write $\Bbbk[\mathbf z]_{\mathbf a}$ for the $\Bbbk$ vector space which has as a basis the monomials of degree $\mathbf a$, \[\{\mathbf z^{\mathbf v}:{\tt d}(\mathbf z^{\mathbf v})={\mathbf a}\}.\]  As a vector space, $\displaystyle\Bbbk[\mathbf z]=\bigoplus_{\mathbf a\in\mathcal A}\Bbbk[\mathbf z]_{\mathbf a}$.
A $\Bbbk[\mathbf z]$-module $M$ is multigraded by $\Bbbk[\mathbf z]$ if it has a direct sum decomposition $\displaystyle M=\bigoplus_{\mathbf a\in \mathbb N^n} M_{\mathbf a}$ which satisfies
\[\Bbbk[\mathbf z]_{\mathbf a}\cdot M_{\mathbf b}\subseteq  M_{\mathbf a+\mathbf b}\] for all $\mathbf a,\mathbf b\in   \mathbb Z^n$.
We will assume that the multigrading is {\bf positive}, that is  each of the graded pieces of $\Bbbk[\mathbf z]$ are finite dimensional as $\Bbbk$ vector spaces.

Let $\mathcal C$ be a function from finitely generated, graded $S$ modules to $\mathbb Z[x_1,\ldots, x_n]$.  We say $\mathcal C$ is {\bf additive} if for each $M$
\begin{equation}
\label{eqn:multidegreeadd}
\mathcal C(M;\mathbf x)=\sum_{i=1}^k {\tt mult}(M,\mathfrak p_i) \mathcal C(\Bbbk[\mathbf z]/\mathfrak p_i;\mathbf x).
\end{equation}
Here, $\{\mathfrak p_1,\ldots, \mathfrak p_k\}$ is the set of maximal dimensional associated primes of $M$ and ${\tt mult}(M,\mathfrak p)$ is the \emph{multiplicity} of $M$ at $\mathfrak p$.  See \cite[Section~3.6]{eisenbud1995commutative}.

Fix a monomial term order on $\Bbbk[\mathbf z]$.  If $f\in \Bbbk[\mathbf z]$, write ${\tt init}(f)$ for its lead term.  The {\bf initial ideal} of $I$ is 
\[{\tt init}(I):=\{{\tt init}(f):f\in I\}.\]   
 $\mathcal C$ is {\bf degenerative} if given a graded free presentation $F/K$, we have 
\begin{equation}
\label{eqn:multidegreedegen}
\mathcal C(F/K;\mathbf x)=\mathcal C(F/{\tt init}(K);\mathbf x).
\end{equation} 
Write $\langle \mathbf a,\mathbf x\rangle:=a_1x_1+a_2x_2+\ldots + a_nx_n$.
\begin{theorem}{{\cite[Theorem 8.44]{miller2004combinatorial}}}
There is a unique function $\mathcal C$ which is additive and degenerative so that 
\begin{equation}
\mathcal C(\Bbbk[\mathbf z]/\langle z_{i_1},\ldots, z_{i_k}\rangle;\mathbf x)= \prod_{\ell=1}^k \langle {\tt d}(z_{i_\ell}),\mathbf x \rangle .
\end{equation}
\end{theorem}
$\mathcal C(M;\mathbf x)$ is called the {\bf multidegree} of $M$.

\begin{lemma}
\label{lemma:summaxfacet}
Suppose $I$ is a square-free monomial ideal in $\Bbbk[\mathbf z]$.  Then
\[\mathcal C(\Bbbk[\mathbf z]/I;\mathbf x)=\sum_{f\in F_{\tt max}(\Delta(I))} C(\Bbbk[\mathbf z]/\mathfrak m_{\overline{f}};\mathbf x).\]
\end{lemma}
\begin{proof}
 Since $I$ is square-free, it has the prime decomposition
\[I=\bigcap_{f\in F(\Delta(I))}\mathfrak m_{\overline{f}}.\]
Squarefree monomial ideals are radical, and so
\[{\tt mult}(\Bbbk[\mathbf z]/I,\Bbbk[\mathbf z]/\mathfrak m_{\overline{f}})=1\] whenever $f\in F(\Delta(I))$. The maximal dimensional associated primes of $\Bbbk[\mathbf z]$ are 
\[\{\mathfrak m_{\overline f}: f\in F_{\tt max}(\Delta(I)) .\] 
Applying additivity,
\[\mathcal C(\Bbbk[\mathbf z]/I;\mathbf x)=\sum_{ f\in F_{\tt max}(\Delta(I))} \mathcal C(\Bbbk[z])/\mathfrak m_{\overline f};\mathbf x). \qedhere\]
\end{proof}

\subsection{ASM varieties}
 Recall ${\sf Mat}(n)$ is the space of $n\times n$ matrices over an algebraically closed field $\Bbbk$. 
Write ${\sf GL}(n)$ for the invertible matrices in ${\sf Mat}(n)$ and ${\sf T}$ for the torus of diagonal matrices in ${\sf GL}(n)$.  There is a natural action of ${\sf GL}(n)$, and hence ${\sf T}$, on ${\sf Mat}(n)$ by left multiplication.  A variety $X\subseteq {\sf Mat}(n)$ is ${\sf T}$ stable if $T\cdot X\subseteq X$.

Let $Z=(z_{ij})_{i,j=1}^n$ be a matrix of generic variables and write $\Bbbk[Z]=\Bbbk[z_{11},z_{12},\ldots, z_{nn}]$ for the coordinate ring of ${\sf Mat}(n)$.  Let ${\tt d}$ be the degree map defined by ${\tt d}(z_{ij})=i$.  
   The multigrading defined by the degree map corresponds to the action of ${\mathsf T}$ on ${\sf Mat}(n)$.  In particular,  ${\sf T}$ stable subvarieties of ${\sf Mat}(n)$ have coordinate rings that are $\Bbbk[Z]$-graded modules.
When $\Bbbk[Z]/I$ is the coordinate ring of $X\subseteq {\sf Mat}(n)$,  write $\mathcal C(X;\mathbf x):=\mathcal C(\Bbbk[Z]/I;\mathbf x)$.  In this situation, (\ref{eqn:multidegreeadd}) becomes
\begin{equation}
\label{eqn:multidegreeaddvariety}
\mathcal C(X;\mathbf x)=\sum_{i=1}^k \mathcal C(X_i;\mathbf x)
\end{equation}
where $\{X_1,\ldots,X_k\}$ are the maximal dimensional irreducible components of $X$.  Since $I$ is radical, (\ref{eqn:multidegreeaddvariety}) is a multiplicity free sum.

 Given an $n\times n$ matrix $M$, write $M_{[i],[j]}$ for the submatrix of $M$ which consists of the first $i$ rows and $j$ columns of $M$.  
Fix $w\in {\sf P}(n)$. The {\bf matrix Schubert variety} is 
\begin{equation}
\label{eqn:matrixschubert}
X_w:=\{M\in M_{[i],[j]}: {\rm rank}(M_{[i],[j]})\leq r_w(i,j) \text{ for all } 1\leq i,j\leq n\}.
\end{equation}
Matrix Schubert varieties generalize classical determinantal varieties.  They were studied by W.~Fulton, who showed that they are irreducible \cite{fulton1992flags}.

Given $A\in {\sf ASM}(n)$, we define the {\bf alternating sign matrix variety} \[X_A:=\{M\in M_{[i],[j]}: {\rm rank}(M_{[i],[j]})\leq r_A(i,j) \text{ for all } 1\leq i,j\leq n\}.\]  
Immediately by definition,
\begin{equation}
\label{eqn:containasmvar}
\text{ if } \quad A\leq B  \quad \text{ then } \quad X_A\supseteq X_B.
\end{equation}

Let ${\sf B}_-,{\sf B}\subset {\sf GL}_n$ be the Borel subgroups of lower triangular and upper triangular matrices respectively.  There is a left action of ${\sf B}_-\times {\sf B}$  on ${\sf Mat}(n)$ given by 
\begin{equation}
(b_1,b_2)\cdot M:=b_1 M b_2^{-1}.
\end{equation}
Write $\Omega_w={\sf B}_-\times {\sf B}\cdot w$ for the orbit through the partial permutation $w$. 
We recall some facts about ${\sf B}_-\times {\sf B}$ orbits.  See \cite[Chapter 15]{miller2004combinatorial}.
\begin{proposition}
\label{prop:orbitclosure}
\begin{enumerate}
\item If $M\in \Omega_w$, then ${\rm rank}(M_{[i],[j]})=r_w(i,j)$ for all $1\leq i,j\leq n$.
\item There is a unique $w\in {\sf P}(n)$ in each ${\sf B}_-\times\sf{B}$ orbit. 
\item $X_w=\overline{\Omega_w}$.  Furthermore, $X_w$ is irreducible and has dimension $n^2-\ell(w)$.
\end{enumerate} 
\end{proposition}
\begin{proof}

\noindent (1) The action of ${\sf B}_-\times{\sf B}$ on $M\in {\sf Mat}(n)$ is by row operations which sweep downwards and column operations which sweep to the right.  Restricted to $M_{[i],[j]}$ this action is just row and column operations within $M_{[i],[j]}$.  So ${\rm rank}(M_{[i],[j]})$ is stable under this action for all $1\leq i,j\leq n$.  In particular, ${\rm rank}(w_{[i],[j]})=r_w(i,j)$, so the result follows.

\noindent (2) This is Proposition~15.27 in \cite{miller2004combinatorial}.  

\noindent (3) See Theorem~15.31 \cite{miller2004combinatorial}.
\end{proof}

$X_A$ has the following set theoretic descriptions as unions and intersections of other ASM varieties.
\begin{proposition}
\label{prop:intersectandunion}
\begin{enumerate}
\item $\displaystyle X_A=\bigcup_{w\in {\tt Perm}(A)} X_w$.
\item If $  A=\vee\{A_1,\ldots,A_k\}$, then $\displaystyle X_A=\bigcap_{i=1}^k A_i$.
\end{enumerate}
\end{proposition}
\begin{proof}
\noindent (1) 
\noindent $(\subseteq)$ Fix $M\in X_A$.  Then
$M\in \Omega_w$ for some $w\in {\sf P}(n)$ and $r_w\leq r_A$.  By Corollary~\ref{cor:perminfty}, there exists $w'\in {\tt Perm}(A)$ so that $w\geq w'$.  So $M\in X_{w}\subseteq X_{w'}$.  Hence $M\in \bigcup_{w\in {\tt Perm}(A)} X_w$.

\noindent $(\supseteq)$ If $w\in {\tt Perm}(A)$ then $w\geq A$.  So by (\ref{eqn:containasmvar}), $X_A\supseteq X_{w}$.  Therefore, \[X_A\supseteq \bigcup_{w\in {\tt Perm}(A)}X_w.\]

\noindent (2)
\noindent $(\subseteq)$ $A\geq A_i$ for all $i$.  So by (\ref{eqn:containasmvar}), $X_A\subseteq \bigcap_{i=1}^k X_{A_i}$.

\noindent $(\supseteq)$  Take $M\in \bigcap_{i=1}^k X_{A_i}$.  Then $M\in \Omega_w$ for some $w\in {\sf P}(n)$.  Since $M\in X_{A_i}$ for all $i$, we have $w\geq A_i$ for all $i$.  So $w\geq A=\vee\{A_1,\ldots, A_k\}$.  Therefore $M\in\Omega_w\subseteq X_A$.
\end{proof}

W.~Fulton showed that each $X_w$ is defined by a smaller set of {\bf essential} conditions,
\begin{equation}
\label{eqn:matrixschubertess}
X_w=\{M\in M_{[i],[j]}: {\rm rank}(M_{[i],[j]})\leq r_w(i,j) \text{ for all } (i,j)\in\mathcal Ess(w)\}.
\end{equation}
By Proposition~\ref{prop:intersectandunion}, $X_A=\bigcap_{u\in {\tt biGr}(A)} X_u$.  Therefore, ASM varieties are also defined by essential conditions.
\begin{equation}
\label{eqn:asmess}
X_A=\{M\in M_{[i],[j]}: {\rm rank}(M_{[i],[j]})\leq r_A(i,j) \text{ for all } (i,j)\in\mathcal Ess(A)\}.
\end{equation}

The rank of \emph{any} submatrix is preserved under the action of ${\sf T}$, so $X_w$ is ${\sf T}$ stable.   By \cite[Theorem~A]{knutson2005grobner}, when $w\in \mathcal S_n$,
\begin{equation}
\label{eqn:matschubmulti}
\mathcal C(X_w;\mathbf x)=\mathfrak S_w.
\end{equation}  

\begin{proposition}
\label{prop:cohosum}
$\displaystyle \mathcal C(X_A;\mathbf x)=\sum_{w\in {\tt MinPerm}(A)} \mathfrak S_w$.
\end{proposition}
\begin{proof}

As a consequence of Proposition~\ref{prop:intersectandunion}, the top dimensional irreducible components of $X_A$ are $\{X_w:w\in {\tt MinPerm}(A)\}$.  Then using the additivity property of multidegrees and (\ref{eqn:matschubmulti}), we have 
\begin{align*}
\mathcal C(X_A;\mathbf x) =\sum_{w\in{\tt MinPerm}(A)} \mathcal C(X_w;\mathbf x) 
= \sum_{w\in{\tt MinPerm}(A)}  \mathfrak S_w. & \qedhere
\end{align*}
\end{proof}
Theorem~\ref{theorem:main} follows as an immediate consequence of  Proposition~\ref{prop:cohosum} and Theorem~\ref{theorem:schubertsum}.

\subsection{Northwest rank conditions}
It is possible to consider more general rank conditions than those defined by corner sums of ASMs.
Let $\mathbf r=(r_{ij})_{i,j=1}^n$ with $r_{ij}\in \mathbb Z_{\geq 0}\cup \{\infty\}$.  The {\bf northwest rank variety} is 
\begin{equation}
X_{\mathbf r}:=\{M\in M_{[i],[j]}: {\rm rank}(M_{[i],[j]})\leq r_{ij} \text{ for all } 1\leq i,j\leq n\}.
\end{equation}

W.~Fulton showed that $X_{\mathbf r}$ is irreducible if and only if $\mathbf r=r_w$ for some $w\in {\sf P}(n)$.  It is stable under the ${\sf B}_-\times {\sf B}$ orbit, so decomposes as a union of (partial) matrix Schubert varieties.  Z. Xu-an and G. Hongzhu classified northwest rank varieties and  gave an algorithm to decompose them into their  irreducible components \cite{zhao2008irreducible}.  

We give an alternative discussion using the order theoretic properties of partial ASMs.
\emph{A priori}, $X_{\mathbf r}$ appears to be a more general object than an ASM variety.
We will show, up to an affine factor, $X_{\mathbf r}$ is isomorphic to some ASM variety.  Furthermore, $X_{\mathbf r}=X_{r_A}$ for some $A\in {\sf PA}(n)$.  So northwest rank varieties are indexed by partial ASMs.
\begin{lemma}
Let $A\in {\sf PA}(n)$ and $\widetilde{A}\in {\sf ASM}(N)$ its completion to an honest ASM.  Then \[X_A\times \Bbbk^{N^2-n^2}\cong X_{\widetilde{A}}.\]
\end{lemma}

\begin{proof}
By construction, $r_{\widetilde A}(i,j)=r_A(i,j)$ for all $1\leq i,j\leq n$.  Therefore, if $M\in X_{\widetilde A}$ then $M_{[n],[n]}\in X_A$.
Conversely, fix $L\in X_A$.  We have $L\in \Omega_w$ for some $w\in {\sf P}(n)$, with $w\geq A$.  Let $L'$ be any matrix in ${\sf Mat}(N)$ so that $L'_{[n],[n]}=L$. Then $L'\in \Omega_v$ for some $v\in {\sf P}(N)$.
Consider the completions $\widetilde{w},\widetilde{v}\in {\sf ASM}(\infty)$. Since $A\leq w$, we have $\widetilde{A}\leq \widetilde{w}\in {\sf ASM}(\infty)$. 

 By construction, $\widetilde{w}$ is the minimum among elements of ${\sf ASM}(\infty)$ which restrict to $w$ in ${\sf P}(n)$.  Since $v_{[n],[n]}=w$, we have  $\widetilde{A}\leq \widetilde{w}\leq \widetilde{v}\in {\sf ASM}(\infty)$.  Then 
$\widetilde{A}\leq \widetilde{v}_{[N],[N]}=v\in {\sf PA}(N)$.
Therefore, $L'\in \Omega_v\subseteq X_{\widetilde{A}}$.
As such, $X_{\widetilde{A}}\cong \{L\times \Bbbk^{N^2-n^2}:L\in X_A\}$.
\end{proof}

Fix a rank function $\mathbf r=(r_{ij})_{i,j=1}^n.$   Let 
\begin{equation}
A_{\mathbf r}=\vee \{[i,j,r_{ij}]_b: r_{ij}<n \}\in {\sf PA}(n).
\end{equation}
\begin{proposition}
$X_{A_{\mathbf r}}=X_{\mathbf r}$.
\end{proposition}
\begin{proof}
If $r_{ij}\geq n$, it is a vacuous rank condition on matrices in ${\sf Mat}(n)$.  So we ignore these entries of $\mathbf r$.  By definition, $X_{A_{\mathbf r}}=\bigcap X_{[i,j,r_{ij}]_b}$ with the intersection taken over $(i,j)$ indexing nonvacuous rank conditions.

If $M\in X_{\mathbf r}$ we have $M\in X_{[i,j,r_{ij}]_b}$ for all $1\leq i,j\leq n$.  So $M\in X_{A_{\mathbf r}}$. 
 Conversely, if $M\in X_{A_{\mathbf r}}$, then ${\rm rank}(M_{[i,j}])\leq r_A(i,j)\leq r_{ij}$ whenever $r_{ij}$ is a nonvacuous rank condition.  So $M\in X_{\mathbf r}$.
\end{proof}

Notice that unions of matrix Schubert varieties need not be northwest rank varieties. 
\begin{example}
Let $X=X_{132}\cup X_{213}$.  If $X=X_{\mathbf r}$, then $r_{132},r_{213}\leq \mathbf r$.  So $r_{132}\wedge r_{213}\leq \mathbf r$.  But $r_{132}\wedge r_{213}=r_{123}$.  Since ${\tt dim}(X_{123})>{\tt dim}(X)$, it follows that $X$  can not be defined by a list of northwest rank conditions. 
 \qed
\end{example}

\subsection{ASM determinantal ideals}
\label{subsection:asmdet}

We now turn our discussion to defining ideals for ASM varieties.   
Define the {\bf ASM ideal} by
\begin{equation}
\label{def:detIdeal}
I_A:=\langle  \text{ minors of size } r_A(i,j)+1 \text{ in } Z_{[i],[j]}\rangle.
\end{equation}
A matrix has rank at most $r$ if and only if all of its minors of size $r+1$ vanish.  As such, $I_A$ set-theoretically cuts out $X_A$. Furthermore, $I_A$ has generators which are homogeneous for the $\mathbb Z^n$ grading on $\Bbbk[Z]$.

\begin{lemma}
\label{lemma:rankIdealContain}
\begin{enumerate}
\item If $r_A\leq r_B$ then $I_A\supseteq I_B$.
\item $\displaystyle I_A=\sum_{u\in {\tt biGr}(A)}I_u=\langle  \text{ minors of size } r_A(i,j)+1 \text{ in } Z_{[i],[j]}:(i,j)\in \mathcal Ess(A)\rangle $.
\end{enumerate}
\end{lemma}
\begin{proof}
\noindent (1) 
Define 
\begin{equation}
I_{i,j}^r=\langle \text{ minors of size } r+1  \text{ in } Z_{[i],[j]}\rangle.
\end{equation}
We may compute each minor by iteratively doing row expansions.  As such,
\begin{equation}
\text{ if } r\leq r' \text{ then  } I_{i,j}^r\supseteq I_{i,j}^{r'}.  
\end{equation}
So suppose $r_A\leq r_B$.  Then
\[I_A=\sum_j\sum_i I_{i,j}^{r_A(i,j)}\supseteq \sum_j\sum_i I_{i,j}^{r_B(i,j)}=I_B.\]

\noindent (2) For each $(i,j)$ there is $u\in {\tt biGr}(A)$ so that $r_A(i,j)=r_u(i,j)$.  As such, $I_{i,j}^{r_A(i,j)}\subseteq I_u$ for some $u$ in ${\tt biGr}(A)$.  By part (1), $I_u\subseteq I_A$, for all $u\in {\tt biGr}(A)$.  Therefore
\begin{align*}
I_A&=\sum_j\sum_i I_{i,j}^{r_A(i,j)}\subseteq \sum_{u\in {\tt biGr}(A)} I_u \subseteq I_A. \qedhere
\end{align*}
\end{proof}
To distinguish between the two generating sets of $I_A$, we refer to 
\[{\tt Gen}(A)=\{\text{ minors of size } r_A(i,j)+1 \text{ in } Z_{[i],[j]}\}\] 
as the {\bf defining generators} of $I_A$.  Call 
\[{\tt EssGen}(A)=\{\text{ minors of size } r_A(i,j)+1 \text{ in } Z_{[i],[j]}:(i,j)\in\mathcal Ess(A)\}\]  the {\bf essential generators} of $I_A$.

\begin{example}
Let $A$ be as in Example~\ref{example:noneqi}.
We have $\mathcal Ess(A)= \{(1,2),(2,3)\}$.  Furthermore, $r_A(1,2)=0$ and $r_A(2,3)=1$.  Applying Lemma~\ref{lemma:rankIdealContain} yields
\begin{align*}
I_A&=\langle z_{11},z_{12} , 
\left |\begin{array}{cc}
z_{11}&z_{12}\\
z_{21}&z_{22}\\
\end{array} \right |,
\left |\begin{array}{cc}
z_{11}&z_{13}\\
z_{21}&z_{23}\\
\end{array}\right | ,
\left |\begin{array}{cc}
z_{12}&z_{13}\\
z_{22}&z_{23}\\
\end{array}\right |
\rangle\\
&=\langle z_{11},z_{12,},z_{13}z_{21},z_{13}z_{22}\rangle\\
&=\langle z_{11},z_{12},z_{21},z_{22}\rangle\cap \langle z_{11},z_{12},z_{13}\rangle\\
&= I_{3412}\cap I_{4123}.  
\end{align*}
This agrees with the irreducible decomposition $X_A=X_{3412}\cup X_{4123}$.
Notice by additivity, $\mathcal C(\Bbbk[Z]/I_A;\mathbf x)=\mathcal C(\Bbbk[Z]/I_{4123})=x_1^3=\mathfrak S_{4123}$.  
\qed
\end{example}

An {\bf antidiagonal} term order on $\Bbbk [Z]$ is a term order for which the lead term of any minor in $Z$ is the product of its antidiagonal terms.  From now on, fix an antidiagonal term order $\prec$ on $\Bbbk[Z]$.
A {\bf Gr\"obner basis} for  $I$ is a set $\{g_1,\ldots, g_k: g_i\in \Bbbk[Z]\}$ so that 
\begin{enumerate}
\item $I=\langle g_1,\ldots, g_k \rangle$, and
\item ${\tt init}( I)=\langle {\tt init}(g_1),\ldots, {\tt init}(g_k) \rangle.$
\end{enumerate}

\begin{proof}[Proof of Proposition~\ref{proposition:maindegeneration}]
\noindent (1) If $w\in \mathcal S_n$, by Section~7.2 of \cite{knutson2009frobenius}, there is a Frobenius splitting for which $X_w$ is compatibly split.  Since $X_A=\bigcap_{u\in {\tt biGr}(A)} X_u$, it is also compatibly split.  

By the argument in \cite{sturmfels1990grobner}, ${\tt EssGen}(u)$ is a Gr\"obner basis for $I_u$.  Since $\mathcal B_n$ is the base of ${\sf ASM}(n)$, we may apply part (2) of \cite[Theorem~6]{knutson2005grobner}, since \[{\tt EssGen}(A)=\bigcup_{u\in {\tt biGr}(A)}{\tt EssGen}(u),\] it is a Gr\"obner basis for $I_A$. Since ${\tt Gen}(A)\supseteq {\tt EssGen}(A)$, we have that ${\tt Gen}(A)$ is also a Gr\"obner basis for $I_A$.

\noindent (2)  The lead terms of ${\tt EssGen}(A)$ are are square-free, hence ${\tt init}(I_A)$ is radical.  Since $I_A$ degenerates to a radical ideal, it is itself radical.

\noindent (3) 
By \cite[Theorem~B]{knutson2005grobner}, if $w\in S_n$,
\begin{equation}
\label{eqn:complex}
\Delta({\tt init}(I_w))=\Delta(Q_{n\times n},w).
\end{equation}

From part (1), \[{\tt init}(I_A)=\sum_{u\in {\tt biGr}(A)}{\tt init}(I_u).\]  Therefore, 
\begin{align*}
\Delta({\tt init}(I_A))&=\bigcap_{u\in {\tt biGr}(A)}\Delta({\tt init}(I_u)) &\text{(by Lemma~\ref{lemma:SRComplex})}\\
&=\bigcap_{u\in {\tt biGr}(A)} \Delta(Q_{n\times n},u) & \text{(by (\ref{eqn:complex}))}\\
&= \Delta(Q_{n\times n},A) & (\text{by part (2) of Proposition~\ref{prop:interesctunion}}.) &\qedhere
\end{align*}
\end{proof}
The discussion in \cite{knutson2009frobenius} assumes $\Bbbk=\mathbb Q$.  However, since the defining generators of $I_A$ have coefficients in $\{\pm 1\}$, the generators are actually Gr\"obner over $\mathbb Z$, and so the statement holds more generally.
Applying Lemma~\ref{lemma:summaxfacet}, we can also compute $\mathcal C(X_A;\mathbf x)$ as the weighted sum over \[F_{\tt max}(\Delta({\tt init} (I_A)))=F_{\tt max}(\Delta_A).\]   Theorem~\ref{thm:wtbijection} gives a weight preserving bijection between ${\tt Prism}(\boldsymbol \lambda,\mathbf d)$ and $F_{\tt max}(\Delta_{A_{\boldsymbol \lambda,\mathbf d}})$. This produces a specific connection between the Gr\"obner geometry of $X_{A_{\boldsymbol \lambda,\mathbf d}}$ and prism tableaux.

\section*{Acknowledgements}  I thank my advisor, Alexander Yong, for his guidance throughout this project.
I also thank Allen Knutson for suggesting this direction of research and Jessica Striker for helpful conversations about alternating sign matrices. 
  I was supported by  a UIUC Campus Research Board and by an NSF Grant.  This work was partially completed while participating in the trimester ``Combinatorics and Interactions'' at the Institut Henri Poincar\'e.  My travel support  was provided by NSF Conference Grant 1643027.  I was funded by the Ruth~V.~Shaff and Genevie~I.~Andrews Fellowship.   I used Sage and Macaulay2 during the course of my research.

\bibliographystyle{mybst}
\bibliography{mylib}
\end{document}